\documentclass[12pt,reqno]{amsart}
\usepackage[all,2cell]{xy}
\usepackage{graphicx,xcolor}

\usepackage{amssymb}
\newtheorem{theorem}{Theorem}[section]
\newtheorem{proposition}[theorem]{Proposition}
\newtheorem{lemma}[theorem]{Lemma}
\newtheorem{corollary}[theorem]{Corollary}

\theoremstyle{definition}
\newtheorem{definition}[theorem]{Definition}
\newtheorem{remark}[theorem]{Remark}

\numberwithin{equation}{section}

\begin{document}
\baselineskip=15pt

\title[Branched ${\rm SL}(r, {\mathbb C})$-opers]{Branched ${\rm SL}(r, {\mathbb C})$-opers}

\author[I. Biswas]{Indranil Biswas}

\address{School of Mathematics, Tata Institute of Fundamental
Research, Homi Bhabha Road, Mumbai 400005, India}

\email{indranil@math.tifr.res.in}

\author[S. Dumitrescu]{Sorin Dumitrescu}

\address{Universit\'e C\^ote d'Azur, CNRS, LJAD, France}

\email{dumitres@unice.fr}

\author[S. Heller]{Sebastian Heller}

\address{Institute of Differential Geometry,
Leibniz Universit\"at Hannover,
Welfengarten 1, 30167 Hannover}
\email{seb.heller@gmail.com}



\date{}

\begin{abstract}
Branched projective structures were introduced by Mandelbaum \cite{Ma1}, \cite{Ma2}, and 
opers were introduced by Beilinson and Drinfeld \cite{BD1}, \cite{BD2}. We define the 
branched analog of ${\rm SL}(r, {\mathbb C})$--opers and investigate their properties. For 
the usual ${\rm SL}(r, {\mathbb C})$--opers, the underlying holomorphic vector bundle is
actually determined uniquely up to tensoring with a holomorphic line bundle
of order $r$. For the branched 
${\rm SL}(r, {\mathbb C})$--opers, the underlying holomorphic vector bundle depends more 
intricately on the oper. While the holomorphic connection for a branched ${\rm SL}(r, 
{\mathbb C})$--oper is nonsingular, given a branched ${\rm SL}(r, {\mathbb C})$--oper, we 
associate to it a certain holomorphic vector bundle equipped with a logarithmic connection. 
This holomorphic vector bundle in question supporting a logarithmic connection does not 
depend on the branched oper. We characterize the branched ${\rm SL}(r, {\mathbb C})$--opers 
in terms of the logarithmic connections on this fixed holomorphic vector bundle.
\end{abstract}

\maketitle

\tableofcontents

\section{Introduction}\label{se1}

Opers were introduced by Beilinson and Drinfeld \cite{BD1}, \cite{BD2}. For this they were motivated by the
earlier works of Drinfeld and Sokolov \cite{DS1}, \cite{DS2}. Given a semisimple complex Lie group $G$,
a $G$--oper on a compact Riemann surface $X$ consists of
\begin{itemize}
\item a holomorphic principal $G$--bundle $P$ on $X$ equipped with a holomorphic connection $\nabla$, and

\item a holomorphic reduction of the structure group of $P$ to a Borel subgroup $B$ of $G$,
\end{itemize}
such that the reduction to $B$ satisfies an analogue of the Griffiths transversality condition 
with respect to this connection $\nabla$, and moreover the second fundamental form of the 
above reduction of structure group to $B$, for the connection $\nabla$, 
satisfies certain nondegeneracy conditions. Opers are useful in diverse topics, for example 
in geometric Langlands correspondence, nonabelian Hodge theory, some branches of 
mathematical physics, differential equations et cetera; see \cite{BF}, \cite{DFKMMN}, 
\cite{FT}, \cite{FG}, \cite{FG2}, \cite{FB2}, \cite{CS}, \cite{Fr}, \cite{Fr2}, \cite{KSZ}, \cite{MR}, 
\cite{BSY} and references therein.

An ${\rm SL}(2,{\mathbb C})$--oper on a compact Riemann surface $X$ corresponds to a complex 
projective structure on $X$ together with a
theta characteristic on $X$. We recall that a complex projective structure on $X$ is given 
by a covering of $X$ by holomorphic coordinate charts such that all the transition 
functions are restrictions of M\"obius transformations \cite{Gu}. During the early 
seventies, Mandelbaum introduced, and studied, the notion of a branched complex projective 
structure \cite{Ma1}, \cite{Ma2}. Examples of branched complex projective structures are provided 
 by pulling back the usual complex 
projective structures through a holomorphic ramified covering map; see \cite{BiD2} for a 
recent study of branched complex projective structures. For higher dimensional complex 
manifolds a branched analog of the more general notion of holomorphic Cartan geometries 
was introduced in \cite{BiD1}.

Our aim here is to introduce, and investigate, the branched analog of ${\rm SL}(r, 
{\mathbb C})$--opers.

Let $X$ be a compact connected Riemann surface, and let $S\, \subset\, X$ be a finite subset.
The holomorphic cotangent bundle and the structure sheaf of $X$ are denoted by $K_X$ respectively ${\mathcal O}_X$.

A {\it branched} ${\rm SL}(r, {\mathbb C})$--oper on $X$, with branching over $S$, is given by a triple of the
form $(V,\, {\mathcal F},\, D)$, where
\begin{itemize}
\item $V$ is a rank $r$ holomorphic vector bundle on $X$ such that $\bigwedge^r V\,=\, {\mathcal O}_X,$

\item ${\mathcal F}$ is a filtration of holomorphic subbundles
\begin{equation}\label{r1}
0\,=\, F_0\, \subset\, F_1 \, \subset\, F_2 \, \subset\,\cdots \, \subset\, F_{r-1}
\, \subset\, F_r\,=\, V
\end{equation}
with ${\rm rank}(F_i)\,=\, i$, for all $1\, \leq\, i\, \leq\, r$, and

\item $D$ is a holomorphic connection on $V$ such that $D(F_i)\, \subset\, F_{i+1}\otimes K_X$,
for all $1\, \leq\, i\, \leq\, r-1$, and $D$ induces the trivial connection on 
$\bigwedge^r V$,
\end{itemize}
such that for all $1\, \leq\, i\, \leq\, r-1$, the second fundamental form of $F_i$
vanishes exactly on the reduced divisor $S$ (see Definition \ref{def1}).

We recall that for a usual ${\rm SL}(r, {\mathbb C})$--oper, the second fundamental forms of the
filtration ${\mathcal F}$ are required to be nonzero everywhere.

For usual ${\rm SL}(r, {\mathbb C})$--opers, i.e., when $S\,=\, 0$, the underlying
holomorphic vector 
bundle $V$ supporting the holomorphic connection is very special. To explain this,
note that any holomorphic
line bundle $\xi$ on $X$ of order $r$ has a tautological holomorphic connection
$D_\xi$ which is uniquely determined by the condition that the connection on
$\xi^{\otimes r}\,=\,{\mathcal O}_X$ induced by $D_\xi$ has trivial monodromy. For any
usual ${\rm SL}(r, {\mathbb C})$--oper $(V,\, {\mathcal F},\, D)$, note that
$$
(V\otimes\xi,\, {\mathcal F}\otimes\xi,\, D\otimes{\rm Id}_\xi+{\rm Id}_V\otimes D_\xi)
$$
is also a usual ${\rm SL}(r, {\mathbb C})$--oper. If we fix a usual
${\rm SL}(r, {\mathbb C})$--oper
$(V_0,\, {\mathcal F}_0,\, D_0)$ on $X$, then for any usual ${\rm SL}(r, {\mathbb C})$--oper
$(V',\, {\mathcal F}',\, D')$ on $X$, there is a holomorphic line bundle $\xi$ on $X$ of
order $r$ such that $V'\,=\, V_0\otimes\xi$. Moreover, we have
$$
{\mathcal F}'\,=\, {\mathcal F}_0\otimes\xi\, ;
$$
so the filtration for a usual ${\rm SL}(r, {\mathbb C})$--oper also does 
not depend on the ${\rm SL}(r, {\mathbb C})$--oper. In the case of
usual ${\rm SL}(r, {\mathbb C})$--opers,
the underlying holomorphic 
vector bundle $V$ in \eqref{r1} is a jet bundle, and the filtration $\mathcal F$ is the 
natural filtration of a jet bundle.

However, for branched ${\rm SL}(r, {\mathbb C})$--opers 
the underlying vector bundle $V$ supporting the holomorphic connection depends more 
intimately on the oper, although each graded piece $F_{i+1}/F_i$ for the filtration 
${\mathcal F}$ in \eqref{r1} is actually independent of the branched ${\rm SL}(r, {\mathbb 
C})$--oper. More precisely, $F_{i+1}/F_i$ is independent of the branched ${\rm SL}(r, 
{\mathbb C})$--oper if $r$ is odd, and when $r$ is even, the quotients $F_{i+1}/F_i$ are 
determined by the theta characteristic on $X$ associated to the branched oper.

Denote the above quotient line bundle $F_r/F_{r-1}$ by $Q$. We prove that a branched ${\rm 
SL}(r, {\mathbb C})$--oper produces a logarithmic connection on $J^{r-1}(Q)\otimes 
{\mathcal O}_X(-(r-1)S)$, where $J^{r-1}(Q)$ is the $(r-1)$-th jet bundle of $Q$; see 
Proposition \ref{prop3}. Given the logarithmic connection on $J^{r-1}(Q)\otimes {\mathcal 
O}_X(-(r-1)S)$ associated to a branched ${\rm SL}(r, {\mathbb C})$--oper, the branched 
${\rm SL}(r, {\mathbb C})$--oper can actually be recovered using the Hecke 
transformations on $J^{r-1}(Q)\otimes {\mathcal O}_X(-(r-1)S)$
and the induced logarithmic connection on the Hecke transformations. The residues of these
logarithmic connection on $J^{r-1}(Q)\otimes{\mathcal O}_X(-(r-1)S)$, and also the 
local properties of the logarithmic connection, are studied here. Based on these studies, we 
characterize all logarithmic connections on $J^{r-1}(Q)\otimes {\mathcal O}_X(-(r-1)S)$ 
that arise from branched ${\rm SL}(r, {\mathbb C})$--opers; see Theorem \ref{thm1}.

The branched ${\rm SL}(r, {\mathbb C})$-opers are closely related
with the ${\rm SL}(r, {\mathbb C})$-opers with regular
singularity introduced by Beilinson and Drinfeld; this
is explained in Section \ref{se7}, which was kindly communicated to us by Edward Frenkel.

Once the branched ${\rm SL}(r, {\mathbb C})$--opers have been defined, it is straight-forward to
define the branched orthogonal and branched symplectic opers. We have omitted this exercise.

\section{Branched SL($r$)--opers}

As before, $X$ is a compact connected Riemann surface of genus $g$. The holomorphic tangent bundle of $X$
is denoted by $TX$.

Let $V$ be a holomorphic
vector bundle on $X$. A \textit{logarithmic connection} on $V$ singular over
a reduced effective divisor $S$ is a holomorphic differential operator
$$
D\, :\, V\, \longrightarrow\, V\otimes K_X\otimes {\mathcal O}_X(S)
$$
satisfying the Leibniz identity which states that
\begin{equation}\label{e1}
D(fs)\,=\, fD(s)+ s\otimes df
\end{equation}
for any locally
defined holomorphic function $f$ on $X$ and any locally defined holomorphic section $s$ of $V$. If
$S$ is the zero divisor, then $D$ is called a \textit{holomorphic connection}
\cite{At}. Any logarithmic connection on $V$ is integrable (same as flat) because $\Omega^2_X\,=\, 0$.
A criterion of Weil and Atiyah says that $V$ admits a holomorphic connection if and only if the degree of
every indecomposable component of $V$ is zero \cite[p.~203, Theorem 10]{At}, \cite{We}. In particular,
if $V$ is indecomposable, and $\text{degree}(V)\,=\, 0$, then $V$ admits a holomorphic connection.

Let $D$ be a logarithmic connection on $V$. For a holomorphic subbundle
$F\, \subset\, V$, consider the following composition of homomorphisms
\begin{equation}\label{ec1}
F\, \hookrightarrow\, V\, \stackrel{D}{\longrightarrow}\, V\otimes K_X\otimes
{\mathcal O}_X(S)\,\xrightarrow{\ q_0\otimes{\rm Id}_{K_X\otimes {\mathcal O}_X(S)}\ }
(V/F)\otimes K_X\otimes {\mathcal O}_X(S)\, ,
\end{equation}
where $q_0\,:\, V\longrightarrow\, V/F$ is the natural quotient map. From \eqref{e1} it 
follows that this composition of homomorphisms is actually ${\mathcal O}_X$--linear, and
hence it produces a holomorphic section
\begin{equation}\label{e2l}
{\rm SF}(D,\, F)\, \in\, H^0(X,\, \text{Hom}(F,\, V/F)\otimes K_X\otimes {\mathcal
O}_X(S))\,=\, H^0(X,\, F^*\otimes (V/F)\otimes K_X\otimes {\mathcal O}_X(S))\, ,
\end{equation}
which is known as the \textit{second fundamental form} of $F$ for the logarithmic connection $D$.

If $D$ is a holomorphic connection on $V$, then the composition of homomorphisms in
\eqref{ec1} simplifies as
\begin{equation}\label{ec}
F\, \hookrightarrow\, V\, \stackrel{D}{\longrightarrow}\, V\otimes K_X\,\xrightarrow{\ q_0\otimes{\rm Id}_{K_X}\ }
(V/F)\otimes K_X\, ,
\end{equation}
and it gives the following analogue of \eqref{e2l}
\begin{equation}\label{e2}
{\rm SF}(D,\, F)\, \in\, H^0(X,\, \text{Hom}(F,\, V/F)\otimes K_X)\,=\,
H^0(X,\, F^*\otimes K_X\otimes (V/F))\, ,
\end{equation}
which is the \textit{second fundamental form} of $F$ for the holomorphic connection $D$.

Assume that the rank $r$ of $V$ is at least two, and $D$ is, as before, a
holomorphic connection on $V$. Let
$$
0\,=\, F_0\, \subset\, F_1 \, \subset\, F_2 \, \subset\,\cdots \, \subset\, F_{r-1} \, \subset\, F_r\,=\, V
$$
be a filtration of holomorphic subbundles such that 
for all $1\, \leq\, i\, \leq\, r-1$,
\begin{itemize} \item ${\rm rank}(F_i)\,=\, i$, and

\item $D(F_i)\, \subset\, F_{i+1}\otimes K_X$.
\end{itemize}
Consequently, the second fundamental form ${\rm SF}(D,\, F_i)$ in \eqref{e2} satisfies the following condition:
$$
{\rm SF}(D,\, F_i)(F_i)\, \subset\, (F_{i+1}/F_i) \otimes K_X\,\subset\, (V/F_i) \otimes K_X
$$
for all $1\,\leq\, i\,\leq\, r-1$. Therefore, ${\rm SF}(D,\, F_i)$ produces a holomorphic homomorphism
\begin{equation}\label{e3} 
{\rm SF}(D,\, i)\, :\, F_i/F_{i-1} \,\longrightarrow\, (F_{i+1}/F_i) \otimes K_X
\end{equation}
for every $1\, \leq\, i\, \leq\, r-1$.

Assume that there is holomorphic line bundle ${\mathbf L}_0$ on $X$ such that the line bundle
$\text{Hom}(F_i/F_{i-1},\, (F_{i+1}/F_i) \otimes K_X)\,=\, (F_{i+1}/F_i)\otimes (F_i/F_{i-1})^*\otimes K_X$
is holomorphically isomorphic to ${\mathbf L}_0$ for all $1\, \leq\, i\, \leq\, r-1$. This
implies that
\begin{equation}\label{l0}
F_{i+1}/F_i \,=\, (F_i/F_{i-1})\otimes ({\mathbf L}_0\otimes TX)\,=\, (F_i/F_{i-1})\otimes {\mathbf L}\, ,
\end{equation}
where ${\mathbf L}\, :=\,{\mathbf L}_0\otimes TX$. Then we have
\begin{equation}\label{e4}
\det V\, :=\,
\bigwedge\nolimits^r V\,=\, \bigotimes_{j=1}^r F_j/F_{j-1}\,=\, F^{\otimes r}_1\otimes
{\mathbf L}^{^{\otimes\frac{r(r-1)}{2}}}
\,=\, (F_r/F_{r-1})^{\otimes r}\otimes\left({\mathbf L}^{^{\otimes \frac{r(r-1)}{2}}}\right)^*\, .
\end{equation}

Fix a reduced effective divisor
\begin{equation}\label{e5}
S\,:=\, \sum_{k=1}^d x_k
\end{equation}
on $X$ of degree $d\, \geq\, 0$; so $\{x_1,\, \cdots ,\, x_d\}$ are distinct points of $X$. We will
introduce branched $\text{SL}(r,{\mathbb C})$--opers on $X$ with branching over $S$.

When $r$ is an even integer, we assume that $d$ is also even. If $r$ is an even integer, fix a holomorphic
line bundle $\mathcal L$ on $X$ of degree $1 +\frac{d}{2}-g\, \in\, \mathbb Z$ such that
\begin{equation}\label{e6}
{\mathcal L}^{^{\otimes 2}}\,=\, TX\otimes {\mathcal O}_X(S)\, ;
\end{equation}
also, fix a holomorphic isomorphism of ${\mathcal L}^{^{\otimes 2}}$ with $TX\otimes {\mathcal O}_X(S)$.
Note that if $S$ is the zero divisor, then the dual ${\mathcal L}^*$ is a theta characteristic on $X$.

\begin{definition}\label{def1}
A \textit{branched} $\text{SL}(r,{\mathbb C})$--{\it oper} over $X$ with {\it branching} over $S$ is a triple
$$
(V,\, {\mathcal F},\, D)\, ,
$$
where
\begin{itemize}
\item $V$ is a rank $r$ holomorphic vector bundle on $X$ such that $\bigwedge^r V\,=\, {\mathcal O}_X$,

\item ${\mathcal F}$ is a filtration of holomorphic subbundles
\begin{equation}\label{ne1}
0\,=\, F_0\, \subset\, F_1 \, \subset\, F_2 \, \subset\,\cdots \, \subset\, F_{r-1} \, \subset\, F_r\,=\, V
\end{equation}
with ${\rm rank}(F_i)\,=\, i$, for all $1\, \leq\, i\, \leq\, r$, and

\item $D$ is a holomorphic connection on $V$,
\end{itemize}
such that the following five conditions hold:
\begin{enumerate}
\item if $r$ is even, then $F_1\,=\, \left({\mathcal L}^{^{\otimes (r-1)}}\right)^*$, where
$\mathcal L$ is a line bundle as in \eqref{e6}, and
if $r$ is odd, then $F_1\,=\,\left(\left(TX\otimes {\mathcal O}_X(S)\right)^{^{\otimes \frac{r-1}{2}}}\right)^*$,

\item $F_{j+1}/F_j\,\simeq\, F_1\otimes \left(TX\otimes {\mathcal O}_X(S)\right)^{\otimes j}$, for all
$0\, \leq\, j\, \leq\, r-1$,

\item the connection on $\bigwedge^r V\,=\, {\mathcal O}_X$ induced by the connection $D$ on $V$ coincides
with the trivial connection given by the de Rham differential $d$ on ${\mathcal O}_X$,

\item $D(F_i)\, \subset\, F_{i+1}\otimes K_X$, for all $1\, \leq\, i\, \leq\, r-1$, and

\item the homomorphism in \eqref{e3}
$$
{\rm SF}(D,\, i)\, :\, F_i/F_{i-1} \,\longrightarrow\, (F_{i+1}/F_i) \otimes K_X
$$
\begin{equation}\label{rc1}
=\, (F_i/F_{i-1})\otimes TX\otimes {\mathcal O}_X(S)\otimes K_X\,=\, (F_i/F_{i-1})\otimes{\mathcal O}_X(S)
\end{equation}
coincides with the natural inclusion map of $F_i/F_{i-1}$ into $(F_i/F_{i-1})\otimes{\mathcal O}_X(S)$. 
\end{enumerate}
\end{definition}

\begin{remark}\label{rem1}
Note that the above condition that $F_{j+1}/F_j\,=\, F_1\otimes \left(TX\otimes {\mathcal 
O}_X(S)\right)^{\otimes j}$ implies that $(F_{i+1}/F_i)\,=\, (F_i/F_{i-1})\otimes
TX\otimes{\mathcal O}_X(S)$. The fifth condition that ${\rm SF}(D,\, i)$ is the natural
inclusion map coincides with the condition that ${\rm SF}(D,\, i)$ is given by tensoring
with the section of ${\mathcal O}_X(S)$ defined by the constant function $1$ on $X$.
\end{remark}

\begin{remark}\label{rem2}
Comparing \eqref{rc1} and \eqref{l0} we conclude that in Definition \ref{def1},
the holomorphic line bundle ${\mathcal O}_X(S)$ plays the role of ${\mathbf L}_0$ in \eqref{l0}. So
in Definition \ref{def1}, the holomorphic line bundle ${\mathcal O}_X(S)\otimes TX$ plays the
role of ${\mathbf L}$ in \eqref{e4}. The second condition in Definition \ref{def1}, that
$F_{j+1}/F_j\,\simeq\, F_1\otimes \left(TX\otimes {\mathcal O}_X(S)\right)^{\otimes j}$,
is obtained from \eqref{l0}, and the first condition in Definition \ref{def1} is
motivated by \eqref{e4}, because $\bigwedge^r V\,=\, {\mathcal O}_X$.
\end{remark}

\begin{remark}\label{rem-ch}
Any two choices of the holomorphic line bundle $\mathcal L$ in \eqref{e6} differ by tensoring
with a line bundle on $X$ of order two. Any holomorphic line bundles $\xi$ of order two on $X$ has
a unique holomorphic connection $D_\xi$ satisfying the following condition: Any holomorphic isomorphism
between $\xi^{\otimes 2}$ and ${\mathcal O}_X$ takes the connection $D_\xi\otimes {\rm Id}_\xi+{\rm Id}_\xi
\otimes D_\xi$ on $\xi^{\otimes 2}$ to the connection on ${\mathcal O}_X$ given by the de Rham differential
$d$. If $W$ is a holomorphic vector bundle on $X$ equipped with a holomorphic connection $D_W$, then
$D_\xi$ and $D_W$ together produce a holomorphic connection $D_W\otimes {\rm Id}_\xi+{\rm Id}_W\otimes D_\xi$
on $W\otimes\xi$. When $r$ is even,
for any two choices of the holomorphic line bundle $\mathcal L$ in \eqref{e6}, a priori, we have two
different spaces of branched $\text{SL}(r,{\mathbb C})$--opers over $X$, with branching over $S$,
given by Definition \ref{def1}. Using the above observation we conclude that these two spaces are
actually canonically identified.
\end{remark}

When $S$ in \eqref{e5} is the zero divisor, a branched $\text{SL}(r,{\mathbb C})$--oper is a 
usual $\text{SL}(r,{\mathbb C})$--oper \cite{BD1}, \cite{BD2}. In that special (unbranched) 
case, the underlying holomorphic vector bundle $V$ of an $\text{SL}(r,{\mathbb C})$--oper 
does not depend on the oper in the following sense: The vector bundle $V$ is the jet bundle
$J^{r-1}({\mathcal L}^{^{\otimes (r-1)}}\otimes\xi)$, where ${\mathcal L}$ is a fixed
holomorphic line bundle on $X$ such that ${\mathcal L}^2\,=\, TX$ while
$\xi$ is a holomorphic line bundle on $X$ of order $r$.

Contrary to the unbranched case, 
when $\text{degree}(S)\,=\, d\, >\, 0$, the underlying holomorphic vector 
bundle $V$ of a $\text{SL}(r,{\mathbb C})$--oper with branching over $S$ generally depends on the 
branched oper. Nevertheless the successive quotients for the filtration $\mathcal F$ of 
$V$ are clearly independent of the branched 
$\text{SL}(r,{\mathbb C})$--oper.
In Section \ref{se3} we will show that a branched $\text{SL}(r,{\mathbb 
C})$--oper $(V,\, {\mathcal F},\, D)$ with branching over $S$ gives rise to a logarithmic 
connection on a certain holomorphic vector bundle over $X$ obtained by performing Hecke 
transformation on $V$ over the points of $S$. The holomorphic vector bundle in question
turns out to be actually independent of the branched $\text{SL}(r,{\mathbb C})$--oper.

Take a branched $\text{SL}(r,{\mathbb C})$--oper $(V,\, {\mathcal F},\, D)$ over $X$ with branching over $S$.
Let $D^*$ be the holomorphic connection on the dual vector bundle $V^*$ induced by the connection $D$. The filtration
$\mathcal F$ of $V$ in \eqref{ne1}
produces a filtration of $V^*$ by holomorphic subbundles as follows: Consider the dual homomorphisms of the
inclusion maps in \eqref{ne1}
$$
V^*\, =\, F^*_r \,\twoheadrightarrow\, F^*_{r-1}\,\twoheadrightarrow\,\cdots \,\twoheadrightarrow\, F^*_2 \,
\twoheadrightarrow\, F^*_1 \, \twoheadrightarrow\, F^*_0\,=\, 0\, ;
$$
the subbundles of $V^*$ giving the filtration are the kernels of the above projections $V^*\,\twoheadrightarrow\, F^*_i$.
This filtration by holomorphic subbundles of $V^*$ will be denoted by ${\mathcal F}^*$.

\begin{lemma}\label{lem1}
The above triple $(V^*,\, {\mathcal F}^*,\, D^*)$ is a branched ${\rm SL}(r,{\mathbb C})$--oper
over $X$ with branching over $S$.
\end{lemma}

\begin{proof}
As in \eqref{ec}, take a holomorphic subbundle $F\, \subset\, V$. So we have the holomorphic subbundle
$(V/F)^*\, \subset\, V^*$. Note that the quotient bundle $V^*/((V/F)^*)$ is identified with $F^*$. Let
$$
{\rm SF}(D^*,\, (V/F)^*)\, \in\, H^0(X,\, \text{Hom}((V/F)^*,\, V^*/((V/F)^*))\otimes K_X)
$$
$$
=\,H^0(X,\, \text{Hom}((V/F)^*,\, F^*)\otimes K_X)\,=\, H^0(X,\, (V/F)\otimes F^*\otimes K_X)
$$
be the second fundamental form of $(V/F)^*\, \subset\, V^*$ for the connection $D^*$ on $V^*$ (see \eqref{e2}).
Then it is straightforward to check that
\begin{equation}\label{es}
{\rm SF}(D^*,\, (V/F)^*)\,=\, {\rm SF}(D,\, F)\, ,
\end{equation}
where ${\rm SF}(D,\, F)$ is constructed in \eqref{e2}. The lemma is a straightforward consequence of \eqref{es}.
\end{proof}

\subsection{Examples}

We now give some examples of branched ${\rm SL}(r,{\mathbb C})$--oper.

Let
$$
\varpi\, :\, X\, \longrightarrow\, Y
$$
be a nonconstant holomorphic map between compact connected Riemann surfaces such that
all the branch points of $\varpi$ have branch number $1$, and let
$(V,\, {\mathcal F},\, D)$ be a usual ${\rm SL}(r,{\mathbb C})$--oper over $Y$. Then the pullback
$(\varpi^*V,\, \varpi^*{\mathcal F},\, \varpi^*D)$ is a branched ${\rm SL}(r,{\mathbb C})$--oper over $X$.
The branching divisor for $(\varpi^*V,\, \varpi^*{\mathcal F},\, \varpi^*D)$ coincides with
the branching divisor for the map $\varpi$.

Let $\varpi\, :\, X\, \longrightarrow\, {\mathbb C}{\mathbb P}^1$ be any nonconstant holomorphic map
such that all the branch points of $\varpi$ have branch number $1$.
Then pulling back the standard ${\rm SL}(2,{\mathbb C})$--oper on ${\mathbb C}{\mathbb P}^1$ by $\varpi$
we get a branched ${\rm SL}(2,{\mathbb C})$--oper over $X$.

More generally, let $V$ be a rank two holomorphic vector
bundle on $X$ with $\bigwedge^2 V\,=\, {\mathcal O}_X$, and let $D$ be a holomorphic connection on
$V$ such that the induced holomorphic connection on $\bigwedge^2 V$
coincides with the holomorphic connection
on ${\mathcal O}_X$ given by the de Rham differential $d$. Let $L\, \subset\, V$ be a holomorphic line subbundle
satisfying the following two conditions:
\begin{enumerate}
\item the second fundamental form ${\rm SF}(D,\, L)$ is not identically zero, and

\item all the zeros of ${\rm SF}(D,\,L)$ are of order one.
\end{enumerate}
Then $(V,\, L,\, D)$ is a branched ${\rm SL}(2,{\mathbb C})$--oper on $X$. The branching divisor for
$(V,\, L,\, D)$ coincides with the vanishing
divisor of ${\rm SF}(D,\,L)$. The previous example, given by the pull-back of the standard
${\rm SL}(2,{\mathbb C})$--oper on ${\mathbb C}{\mathbb P}^1$, actually corresponds
to the special case where $V\,=\, \varpi^*{\mathcal O}^{\oplus 2}_{{\mathbb C}{\mathbb P}^1}
\,=\, {\mathcal O}^{\oplus 2}_X$, $D$ is the trivial connection on it, and $L$ is the
pull-back, by $\varpi$, of the tautological line subbundle of ${\mathcal O}^{\oplus 2}_{{\mathbb C}{\mathbb P}^1}$.

Let $(V,\, L,\, D)$ be a branched ${\rm SL}(2,{\mathbb C})$--oper over $X$. Then the holomorphic line subbundle
$L\, \subset\, V$ produces a filtration of holomorphic
subbundles $\{F_i\}_{i=1}^r$ of the symmetric product $\text{Sym}^{r-1}(V)$ whose $i$-th
term $F_i$ is the image of $L^{\otimes (r-i)}\otimes V^{\otimes (i-1)}$ in $\text{Sym}^{r-1}(V)$. The triple
$(\text{Sym}^{r-1}(V),\, \{F_i\}_{i=1}^r,\, \widetilde{D})$, where $\widetilde{D}$ is the
holomorphic connection on $\text{Sym}^{r-1}(V)$ induced by $D$, is a branched ${\rm SL}(r,{\mathbb C})$--oper over $X$.
Its branching divisor coincides with that of $(V,\, L,\, D)$.

\section{Logarithmic connection from branched SL($r$)--opers}\label{se3}

\subsection{Homomorphism to a jet bundle}

As in Definition \ref{def1}, let
\begin{equation}\label{gbo}
(V,\, {\mathcal F},\, D)
\end{equation}
be a branched $\text{SL}(r,{\mathbb C})$--oper over $X$, with branching over $S$, where
${\mathcal F}$ stands for the filtration in \eqref{ne1} of holomorphic subbundles of $V$.

In this subsection we will construct another filtered holomorphic vector bundle from the 
branched $\text{SL}(r,{\mathbb C})$--oper $(V,\, {\mathcal F},\, D)$ in \eqref{gbo}.

For notational convenience, let
\begin{equation}\label{e7}
Q\,:=\, F_r/F_{r-1}\,=\, F_1\otimes \left(TX\otimes {\mathcal O}_X(S)\right)^{\otimes (r-1)}
\end{equation}
denote the quotient line bundle in \eqref{ne1}; the second statement in Definition \ref{def1} gives
the above isomorphism. We will construct a holomorphic homomorphism
\begin{equation}\label{e8}
\Phi\, :\, V\, \longrightarrow\, J^{r-1}(Q)
\end{equation}
to the $(r-1)$-th order jet bundle of the line bundle $Q$ in \eqref{e7}; see \cite{Gu}, \cite{BiD2}
for jet bundles.

Let
\begin{equation}\label{e9}
q\, :\, V \, \longrightarrow\,Q\,=\, V/F_{r-1}
\end{equation}
be the natural quotient map. Take any $x\, \in\, X$ and also take any $v\, \in\, V_x$. Consider a simply
connected open neighborhood $x\, \in\, U\, \subset\, X$ of $x$, and denote by
$$
\widetilde{v}\, \in\, H^0(U,\, V)
$$
the unique flat section of $V\big\vert_U$, for the connection $D$, such that $\widetilde{v}(x)\,=\, v$.
Restricting the section $q(\widetilde{v}) \, \in\, H^0(U,\, Q)$ to the $(r-1)$-th order infinitesimal
neighborhood of $x$, where $q$ is the quotient map in \eqref{e9}, we get an element
$$
\widetilde{v}'\, :=\, q(\widetilde{v})\big\vert_{rx} \, \in\, J^{r-1}(Q)_x\, ,
$$
where $rx$ is the nonreduced divisor with multiplicity $r$.
The map $\Phi$ in \eqref{e8} sends $v$ to this element $\widetilde{v}'\, \in\, J^{r-1}(Q)_x$
constructed from $v$ using the connection $D$. The homomorphism $\Phi$ is evidently holomorphic.

We have the natural short exact sequence of jet bundles
\begin{equation}\label{je}
0\,\longrightarrow\,Q\otimes K^{\otimes k}_X\,\longrightarrow\,J^{k}(Q)\,
\longrightarrow\, J^{k-1}(Q) \, \longrightarrow\, 0
\end{equation}
for all $k\, \geq\, 1$. For $0\, \leq\, j\, \leq\, r$, let $H_j$ be the kernel of the projection
$J^{r-1}(Q)\,\longrightarrow\, J^{r-1-j}(Q)$ obtained by iterating the projection in \eqref{je}; we use
the convention that $J^k(W)\,=\, 0$ if $k\, <\, 0$. So ${\rm rank}(H_j)\,=\, j$, and we have the
short exact sequence of holomorphic vector bundles
\begin{equation}\label{e10}
0\, \longrightarrow\, H_j \, \longrightarrow\, J^{r-1}(Q)\,\longrightarrow\, J^{r-1-j}(Q) \, \longrightarrow\, 0
\end{equation}
on $X$. From \eqref{e10} it follows that the quotient $H_j/H_{j-1}$ coincides with the kernel of
the projection $J^{r-j}(Q)\,\longrightarrow\, J^{r-j-1}(Q)$. Therefore, from \eqref{je} it follows that
\begin{equation}\label{e10a}
H_j/H_{j-1}\,=\, Q\otimes K^{\otimes (r-j)}_X
\end{equation}
for all $1\, \leq\, j\, \leq\, r$.

Let
\begin{equation}\label{e11}
X_0\, :=\, X\setminus \{x_1,\, \cdots,\, x_d\}\,=\, X\setminus S\, \subset\, X
\end{equation}
be the complement of $S$ in \eqref{e5}.

The triple
$(V,\, {\mathcal F},\, D)$ in \eqref{gbo} defines a usual $\text{SL}(r,{\mathbb C})$--oper
over the open subset $X_0\, \subset\, X$ in \eqref{e11}.
The following proposition is well-known for the usual $\text{SL}(r,{\mathbb C})$--opers.

\begin{proposition}\label{prop1}
The restriction $\Phi\big\vert_{X_0}$ of the homomorphism in \eqref{e8} to $X_0$
is an isomorphism $V\big\vert_{X_0}\, \stackrel{\sim}{\longrightarrow}\, J^{r-1}(Q)\big\vert_{X_0}$.

For all $0\, \leq\, j\, \leq\, r$, the homomorphism $\Phi\big\vert_{X_0}$ takes the subbundle $F_j\big\vert_{X_0}$
(defined in \eqref{ne1}) isomorphically to the subbundle $H_j\big\vert_{X_0}$ (constructed in \eqref{e10}).
\end{proposition}

{}From Proposition \ref{prop1} it follows that when $S$ is the zero divisor, then
$$
F_r \,=\, V\,=\, J^{r-1}(Q)\, .
$$

The following is a consequence of Proposition \ref{prop1}.

\begin{corollary}\label{cor1}
The homomorphism of ${\mathcal O}_X$--modules $\Phi$ in \eqref{e8} is injective.
It satisfies the condition
$$
\Phi(F_j)\, \subset\, H_j
$$
for all $0\, \leq\, j\, \leq\, r$, where $F_j$ and $H_j$ are defined in \eqref{ne1} and
\eqref{e10} respectively.
\end{corollary}

\begin{proof}
Since $\Phi\big\vert_{X_0}$ is injective it follows immediately that $\Phi$ is injective.
Since $$\Phi\big\vert_{X_0}(F_j\big\vert_{X_0})\, \subset\, H_j\big\vert_{X_0}$$ (see
Proposition \ref{prop1}), and $X_0$ is a dense subset
of $X$, we conclude that $\Phi(F_j)\, \subset\, H_j$.
\end{proof}

{}From Corollary \ref{cor1} it follows that $\Phi$ produces a grading preserving
holomorphic homomorphism from the graded
vector bundle $\bigoplus_{j=1}^r F_j/F_{j-1}$ in \eqref{ne1} to the graded
vector bundle $\bigoplus_{j=1}^r H_j/H_{j-1}$ in \eqref{e10}. For any $1\, \leq\, j\, \leq\, r$, let
\begin{equation}\label{e13}
\Phi_j\, :\, F_j/F_{j-1}\,\longrightarrow\, H_j/H_{j-1}\,=\, Q\otimes K^{\otimes (r-j)}_X
\end{equation}
be the homomorphism induced by $\Phi$; see \eqref{e10a} for the isomorphism in \eqref{e13}.

For any $1\, \leq\, j\, \leq\, r$, note that $Q\otimes K^{\otimes (r-j)}_X\otimes
{\mathcal O}_X(-(r-j)S)$ is actually a subsheaf 
of $Q\otimes K^{\otimes (r-j)}_X$ because $S$ is an effective divisor. The following proposition
identifies this subsheaf $Q\otimes K^{\otimes (r-j)}_X\otimes
{\mathcal O}_X(-(r-j)S)$ with the image of the homomorphism $\Phi_j$.

\begin{proposition}\label{prop2}
For each $1\, \leq\, j\, \leq\, r$, the image of the homomorphism $\Phi_j$ in \eqref{e13}
is the subsheaf
$$
Q\otimes K^{\otimes (r-j)}_X\otimes {\mathcal O}_X(-(r-j)S)\, \subset\,
Q\otimes K^{\otimes (r-j)}_X\,=\, H_j/H_{j-1}\, .
$$
\end{proposition}

\begin{proof}
Recall the homomorphisms ${\rm SF}(D,\, i)$ in statement (5) of Definition \ref{def1}.
Take any $1\, \leq\, k\, \leq\, r-1$. Consider the composition of homomorphisms
$$
({\rm SF}(D,\, r-1)\otimes {\rm Id}_{K^{\otimes(r-1-k)}_X})\circ\cdots\circ ({\rm SF}(D,\, k+1)\otimes
{\rm Id}_{K_X})\circ{\rm SF}(D,\, k)\, :\, F_k/F_{k-1}
$$
\begin{equation}\label{co}
 \longrightarrow\,(F_r/F_{r-1})\otimes K^{\otimes(r-k)}_X
\,=\, Q\otimes K^{\otimes(r-k)}_X\,=\, (F_k/F_{k-1})\otimes {\mathcal O}_X((r-k)S)\, ;
\end{equation}
note that since $(F_{i+1}/F_i)\otimes K_X\,=\, (F_i/F_{i-1})\otimes{\mathcal O}_X(S)$ for all
$1\,\leq\, i\, \leq\, r-1$ (see Remark \ref{rem1}) if follows that
$(F_r/F_{r-1})\otimes K^{\otimes(r-k)}_X\,=\, (F_k/F_{k-1})\otimes {\mathcal O}_X((r-k)S)$.
{}From statement (5) in Definition \ref{def1} we know that the composition of homomorphisms
in \eqref{co} coincides with the natural inclusion map
$$
F_k/F_{k-1}\,\hookrightarrow\, (F_k/F_{k-1})\otimes {\mathcal O}_X((r-k)S)\, .
$$
In other words, if $\textbf{1}\,\in\, H^0(X,\, {\mathcal O}_X((r-k)S))$ is the section
given by the constant function $1$ on $X$, then
the composition of homomorphisms in \eqref{co} coincides with the
homomorphism
$$
{\rm Id}_{F_k/F_{k-1}}\otimes \textbf{1}\, :\, F_k/F_{k-1} \,\longrightarrow\,
(F_k/F_{k-1})\otimes {\mathcal O}_X((r-k)S)\, .
$$
On the other hand, we have $Q\otimes K^{\otimes(r-k)}_X\,=\, H_k/H_{k-1}$ (see \eqref{e10a}). Therefore,
the image of the composition of homomorphisms in \eqref{co} is the subsheaf
$$(H_k/H_{k-1})\otimes
{\mathcal O}_X(-(r-k)S)\,=\,Q\otimes K^{\otimes (r-k)}_X\otimes {\mathcal O}_X(-(r-k)S)\,
\subset\, Q\otimes K^{\otimes (r-k)}_X\, .$$
This completes the proof.
\end{proof}

For $0\, \leq\, j\, \leq\, r$, define
$$
\widehat{H}_j\,:=\, H_j\otimes {\mathcal O}_X(-(r-1)S)\, \subset\, J^{r-1}(Q)\otimes {\mathcal O}_X(-(r-1)S)\, ,
$$
where $H_j$ is constructed in \eqref{e10}. So the filtration $\{H_j\}_{j=0}^r$ of $J^{r-1}(Q)$ produces
a filtration of holomorphic subbundles
\begin{equation}\label{e23}
0\,=\, \widehat{H}_0\, \subset\, \widehat{H}_1\, \subset\, \cdots \, \subset\,
\widehat{H}_{r-1}\, \subset\, \widehat{H}_{r}\,=\, J^{r-1}(Q)\otimes {\mathcal O}_X(-(r-1)S)
\end{equation}
of $J^{r-1}(Q)\otimes {\mathcal O}_X(-(r-1)S)$.

The next result is deduced using Proposition \ref{prop2} and Corollary \ref{cor1}.

\begin{corollary}\label{cor3}
The natural inclusion map $I\, :\, J^{r-1}(Q)\otimes {\mathcal O}_X(-(r-1)S) \,\hookrightarrow\, J^{r-1}(Q)$ factors
through $\Phi$ in \eqref{e8}, meaning there is a unique homomorphism
$${\mathbf i}\, :\, J^{r-1}(Q)\otimes {\mathcal O}_X(-(r-1)S) \, \longrightarrow\, V$$
such that $I\,=\, \Phi\circ {\mathbf i}$. For all $0\, \leq\, j\, \leq\, r$,
$${\mathbf i}(\widehat{H}_j)\,=\, F_j\, ,$$ where $\widehat{H}_j$ and $F_j$ are as in 
\eqref{e23} and \eqref{ne1}, and moreover, $\Phi(F_j) \,=\, H_j$.
\end{corollary}

\begin{proof}
By Proposition \ref{prop2}, the image of the injective homomorphism $\Phi_j$ in \eqref{e13} is the subsheaf
$$
Q\otimes K^{\otimes (r-j)}_X\otimes {\mathcal O}_X(-(r-j)S)\, \subset\,
Q\otimes K^{\otimes (r-j)}_X\,=\, H_j/H_{j-1}\, .
$$
Consequently, we have the following inclusions:
\begin{equation}\label{j1}
\widehat{H}_j/\widehat{H}_{j-1}\,=\,
(H_j/H_{j-1})\otimes {\mathcal O}_X(-(r-1)S) \, \hookrightarrow\,
F_j/F_{j-1}
\end{equation}
$$
=\, (H_j/H_{j-1})\otimes {\mathcal O}_X(-(r-j)S)\, \hookrightarrow\, H_j/H_{j-1}\, ,
$$
where the isomorphism $F_j/F_{j-1}\,=\, (H_j/H_{j-1})\otimes {\mathcal O}_X(-(r-j)S)$ is given by $\Phi_j$.
In view of Corollary \ref{cor1}, from these inclusion maps we conclude that the homomorphism $\Phi$ in
\eqref{e8} satisfies the following:
\begin{equation}\label{f1}
\widehat{H}_r\,:=\,
J^{r-1}(Q)\otimes{\mathcal O}_X(-(r-1)S)\, \stackrel{{\mathbf i}'}{\hookrightarrow}\, \Phi(V) \, \subset\, J^{r-1}(Q)\, .
\end{equation}
Since the coherent analytic sheaf $V$ is identified with the coherent analytic sheaf $\Phi(V)$ using the
isomorphism $\Phi$ between them. the homomorphism ${\mathbf i}'$ produces the homomorphism $\mathbf i$ in the
statement of the corollary, in other words, ${\mathbf i}$ is defined by the equality ${\mathbf i}'\,=\, \Phi\circ{\mathbf i}$.

The map ${\mathbf i}'$
in \eqref{f1} evidently takes $\widehat{H}_j$ in \eqref{e23} into $\Phi(F_j)$. Also,
the inclusion map $\Phi(V) \, \hookrightarrow\, J^{r-1}(Q)$ in \eqref{f1}
takes $\Phi(F_j)$ into the subbundle $H_j\, \subset\, J^{r-1}(Q)$ in \eqref{e10}.
\end{proof}

In view of Corollary \ref{cor3}, the filtration of holomorphic subbundles in \eqref{e23} fits in the
following filtration of coherent analytic subsheaves:
\begin{equation}\label{e23n}
0\,=\, \widehat{H}_0\, \subset\, \widehat{H}_1\, \subset\, \cdots \, \subset\,
\widehat{H}_{r-1}\, \subset\, \widehat{H}_{r}\,=\, J^{r-1}(Q)\otimes {\mathcal O}_X(-(r-1)S)\, \subset\, V\, .
\end{equation}

\subsection{A logarithmic connection}

Take a branched ${\rm SL}(r,{\mathbb C})$--oper $(V,\, {\mathcal F},\, D)$ as in \eqref{gbo}. 
The following proposition shows that the holomorphic connection $D$ produces a logarithmic
connection on the holomorphic vector bundle $J^{r-1}(Q)\otimes {\mathcal O}_X(-(r-1)S)$.

\begin{proposition}\label{prop3}
The holomorphic connection $D\, :\, V\, \longrightarrow\, V\otimes K_X$ in \eqref{gbo}
sends the subsheaf $\widehat{H}_r\, :=\, J^{r-1}(Q)\otimes {\mathcal O}_X(-(r-1)S)\, \subset\, V$ in \eqref{e23n}
to the subsheaf
$$
\big(J^{r-1}(Q)\otimes {\mathcal O}_X(-(r-2)S)\otimes K_X\big)\bigcap (V\otimes K_X)
\, \subset\, V\otimes K_X\, .
$$
In other words, $D$ produces a logarithmic connection, singular over $S$, on the holomorphic
vector bundle $J^{r-1}(Q)\otimes {\mathcal O}_X(-(r-1)S)$.
\end{proposition}

\begin{proof}
Recall that $V\big\vert_{X_0}\,=\, J^{r-1}(Q)\big\vert_{X_0}\,=\, 
(J^{r-1}(Q)\otimes{\mathcal O}_X(-(r-2)S))\big\vert_{X_0}$, where $X_0$ is the open subset in
\eqref{e11}. So we need to investigate the connection $D$ only around the points of $S$. Take any point
$$
x'\, \in\, S\, .
$$
Fix a holomorphic splitting of the filtration $\{F_i\}_{i=0}^r$ of $V$ in \eqref{ne1} over a sufficiently
small analytic neighborhood $U$ of $x'$; this subset $U$ is chosen such that $U\bigcap S\,=\, x'$. So we have
a holomorphic isomorphism
\begin{equation}\label{e24}
V\big\vert_U\,\stackrel{\sim}{\longrightarrow}\,\bigoplus_{i=1}^r (F_i/F_{i-1})\big\vert_U\,=:\,
\bigoplus_{i=1}^r {\mathcal F}_i
\end{equation}
using the notation ${\mathcal F}_i\, :=\, (F_i/F_{i-1})\big\vert_U$; it should be clarified that there
is no natural isomorphism. We would express the
holomorphic connection $D\big\vert_U$ in \eqref{gbo} in terms of the decomposition in \eqref{e24}. From
the fifth condition in Definition \ref{def1} it follows that $D\big\vert_U$ has the following expression
in terms of the decomposition in \eqref{e24}:
\begin{equation}\label{e25}
D\big\vert_U\,=\,
\begin{pmatrix}
D_1& \alpha_{1,2} &\alpha_{1,3} & \alpha_{1,4}&\cdots & \alpha_{1,r-2} & \alpha_{1,r-1}& \alpha_{1,r}\\
\gamma_1 & D_2 & \alpha_{2,3} & \alpha_{2,4} &\cdots & \alpha_{2,r-2} & \alpha_{2,r-1}& \alpha_{2,r}\\
0 & \gamma_2 & D_3 & \alpha_{3,4}& \cdots & \alpha_{3,r-2} & \alpha_{3,r-1}& \alpha_{3,r}\\
0 & 0 & \gamma_3 & D_4 &\cdots & \alpha_{4,r-2} & \alpha_{4,r-1}& \alpha_{4,r}\\
\vdots & \vdots &\vdots & \vdots & \vdots & \vdots & \vdots & \vdots \\
0 & 0 & 0 &0& \cdots & D_{r-2} & \alpha_{r-2,r-1}& \alpha_{r-2,r}\\
0 & 0 & 0 & 0& \cdots & \gamma_{r-2} & D_{r-1} & \alpha_{r-1,r}\\
0 & 0 & 0 & 0& \cdots & 0 & \gamma_{r-1} & D_{r}\\
\end{pmatrix}
\end{equation}
where $D_i$ (the entry at the $i\times i$-th position of the matrix) is a holomorphic connection
on ${\mathcal F}_i$, and $\gamma_i$ (the entry at the $(i+1)\times i$-th position) is a section
\begin{equation}\label{gi}
\gamma_i\, \in\, H^0(U,\, \text{Hom}({\mathcal F}_i,\, {\mathcal F}_{i+1})\otimes (K_X\big\vert_U)
\otimes {\mathcal O}_U(-x'))\, ;
\end{equation}
in other words, $\gamma_i$ is a holomorphic homomorphism ${\mathcal F}_i\, \longrightarrow\,
{\mathcal F}_{i+1}\otimes (K_X\big\vert_U)$ that vanishes at the point $x'\, \in\, U$. For any $j\, >\, i$,
we have
$$
\alpha_{i,j}\, \in\, H^0(U,\, \text{Hom}({\mathcal F}_j,\, {\mathcal F}_i)\otimes (K_X\big\vert_U))\, .
$$
The entry at the $i\times j$-th position of the matrix in \eqref{e25} is zero if $i\, >\, j+1$. It may be
mentioned that by choosing the splitting in \eqref{e24} carefully it is possible to make $\alpha_{i,j}\,=\,0$
for all $j\, >\, i$; but this will not be needed here.

The decomposition of $V\big\vert_U$ in \eqref{e24} produces a holomorphic decomposition of
the vector bundle $(J^{r-1}(Q)\otimes {\mathcal O}_X(-(r-1)S))\big\vert_U$. To see this, consider the
intersection of coherent analytic subsheaves of $V\big\vert_U$
\begin{equation}\label{i}
{\mathcal G}_i\, :=\, {\mathcal F}_i\bigcap\big(J^{r-1}(Q)\otimes {\mathcal O}_X(-(r-1)S)\big)\big\vert_U
\,=\, {\mathcal F}_i\bigcap\big(J^{r-1}(Q\big\vert_U)\otimes {\mathcal O}_U(-(r-1)x')\big)
\, \subset\, V\big\vert_U\, ;
\end{equation}
note that both ${\mathcal F}_i$ and $J^{r-1}(Q\big\vert_U)\otimes {\mathcal O}_U(-(r-1)x')$
are subsheaves of $V\big\vert_U$ (see Corollary \ref{cor3} and \eqref{e24}). Then we have a holomorphic
decomposition of $J^{r-1}(Q\big\vert_U)\otimes {\mathcal O}_U(-(r-1)x')$
\begin{equation}\label{i2}
J^{r-1}(Q\big\vert_U)\otimes {\mathcal O}_U(-(r-1)x')\,=\, \bigoplus_{i=1}^r {\mathcal G}_i
\end{equation}
into a direct sum of holomorphic line bundles on $U$. Indeed, the natural homomorphism of coherent analytic sheaves
$$
\bigoplus_{i=1}^r {\mathcal G}_i\, \longrightarrow\, J^{r-1}(Q\big\vert_U)\otimes {\mathcal O}_U(-(r-1)x')
$$
is clearly surjective; it is also injective because it is injective over the open subset $U\setminus\{x'\}
\, \subset\, U$ and the coherent analytic sheaf $\bigoplus_{i=1}^r
{\mathcal G}_i$ is torsionfree. From \eqref{j1} it follows immediately that
\begin{itemize}
\item ${\mathcal G}_i\,=\, {\mathcal F}_i\otimes {\mathcal O}_U(-(i-1)x')$, and

\item the inclusion map ${\mathcal G}_i\, \hookrightarrow\, {\mathcal F}_i$ (see \eqref{i}) coincides with the
natural inclusion map
$$
{\mathcal F}_i\otimes {\mathcal O}_U(-(i-1)x')) \, \hookrightarrow\, {\mathcal F}_i\, .
$$
\end{itemize}

Now it is straight-forward to check that the connection operator $D\big\vert_U$ in \eqref{e25} on
$\bigoplus_{i=1}^r {\mathcal F}_i$ produces a holomorphic differential operator
$$
\bigoplus_{i=1}^r {\mathcal G}_i
\, \longrightarrow\, \left(\bigoplus_{i=1}^r {\mathcal G}_i\right)\otimes (K_X\big\vert_U)\otimes
{\mathcal O}_U(x')\, .
$$
To see this, first note that the section $\gamma_i$ in \eqref{e25} produces a
holomorphic homomorphism
$$
{\mathcal G}_i\, \longrightarrow\, {\mathcal G}_{i+1}\otimes (K_X\big\vert_U)
$$
because the homomorphism $\gamma_i \in\, H^0(U,\, \text{Hom}({\mathcal F}_i,\,
{\mathcal F}_{i+1})\otimes (K_X\big\vert_U))$ in \eqref{gi} vanishes at the point $x'$.
Secondly, for any $j\, >\, i$, the section $\alpha_{i,j}$ in \eqref{e25} produces a section
\begin{equation}\label{i3}
\widetilde{\alpha}_{i,j}\, \in\, H^0(U,\, \text{Hom}({\mathcal G}_j,\, {\mathcal G}_i)
\otimes (K_X\big\vert_U)\otimes {\mathcal O}_U(-(j-i)x'))\, ,
\end{equation}
in particular, $\widetilde{\alpha}_{i,j}\, \in\, H^0(U,\, \text{Hom}({\mathcal G}_j,\, {\mathcal G}_i)
\otimes (K_X\big\vert_U))$. Thirdly, the connection operator
$$
D_i\, :\, {\mathcal F}_i\, \longrightarrow\, {\mathcal F}_i \otimes (K_X\big\vert_U)
$$
in \eqref{e25} produces a first order holomorphic differential operator
$$
{\mathcal G}_i\, \longrightarrow\, {\mathcal G}_i \otimes (K_X\big\vert_U)\otimes{\mathcal O}_U(x')
$$
that satisfies the Leibniz identity, because $D_i$ itself satisfies the Leibniz identity.

Consequently, $D$ sends the subsheaf $J^{r-1}(Q)\otimes {\mathcal O}_X(-(r-1)S)\, \subset\, V$ in \eqref{e23n}
to the subsheaf
$$
\big(J^{r-1}(Q)\otimes {\mathcal O}_X(-(r-2)S)\otimes K_X\big)\bigcap (V\otimes K_X)
\, \subset\, V\otimes K_X\, .
$$
Since
$$
J^{r-1}(Q)\otimes {\mathcal O}_X(-(r-2)S)\otimes K_X\,=\, J^{r-1}(Q)\otimes {\mathcal O}_X(-(r-1)S)\otimes K_X
\otimes {\mathcal O}_X(S)\, ,
$$
this implies that $D$ produces a logarithmic connection on $J^{r-1}(Q)\otimes {\mathcal O}_X(-(r-1)S)$.
\end{proof}

\begin{remark}\label{rem3}
Consider the holomorphic connection $D_i$ in \eqref{e25} on the holomorphic vector bundle
${\mathcal F}_i$ on $U$. The connection on ${\mathcal G}_i\, :=\,
{\mathcal F}_i\bigcap (J^{r-1}(Q\big\vert_U)\otimes {\mathcal O}_U(-(r-1)x'))$ induced by
$D_i$ is singular at $x'$ if $i\,\geq\, 2$. Therefore, the singular locus of the
logarithmic connections on $J^{r-1}(Q)\otimes {\mathcal
O}_X(-(r-1)S)$ constructed in Proposition \ref{prop3} is exactly $S$.
\end{remark}

Take any branched $\text{SL}(r,{\mathbb C})$--oper $(V,\, {\mathcal F},\, D)$ as in \eqref{gbo}.
Let $\mathcal D$ denote the logarithmic connection on $J^{r-1}(Q)\otimes {\mathcal
O}_X(-(r-1)S)$ constructed in Proposition \ref{prop3} from $(V,\, {\mathcal F},\, D)$.
The following is a straight-forward consequence of the construction of $\mathcal D$.

\begin{corollary}\label{cor4}\mbox{}
\begin{enumerate}
\item The logarithmic connection $\mathcal D$ on $J^{r-1}(Q)\otimes {\mathcal
O}_X(-(r-1)S)$ constructed in Proposition \ref{prop3} satisfies the condition
$$
{\mathcal D}(\widehat{H}_i)\, \subset\, \widehat{H}_{i+1}\otimes K_X\otimes{\mathcal O}_X(S)
$$
for all $1\, \leq\, i\, \leq\, r-1$, where $\{\widehat{H}_j\}_{j=0}^r$ is the filtration of
$J^{r-1}(Q)\otimes {\mathcal O}_X(-(r-1)S)$ in \eqref{e23}.

\item The second fundamental form of the subbundle $$\widehat{H}_i\big\vert_{X_0}\, \subset\,
(J^{r-1}(Q)\otimes {\mathcal O}_X(-(r-1)S))\big\vert_{X_0}$$ for the holomorphic connection
${\mathcal D}\big\vert_{X_0}$ is nowhere zero on $X_0$ for all $1\, \leq\, i\, \leq\, r-1$,
where $X_0$ is the open subset in \eqref{e11}.
\end{enumerate}
\end{corollary}

\begin{proof}
Over the open subset $X_0$ in \eqref{e11}, we have
$$
J^{r-1}(Q)\big\vert_{X_0}\,=\, J^{r-1}(Q)\otimes {\mathcal O}_X(-(r-1)S)\big\vert_{X_0}\,
=\, V\big\vert_{X_0}
$$
(see Proposition \ref{prop1}). This isomorphism takes $F_j\big\vert_{X_0}\, \subset\, 
V\big\vert_{X_0}$ isomorphically to $H_j\big\vert_{X_0}\,=\, \widehat{H}_j\big\vert_{X_0}$ 
(see Proposition \ref{prop1}). It also takes the holomorphic connection $D\big\vert_{X_0}$ 
to ${\mathcal D}\big\vert_{X_0}$; this is an immediate consequence of the construction of 
${\mathcal D}$. Since $X_0$ is dense in $X$, the first statement now follows from the facts 
that $D(F_i)\, \subset\, F_{i+1}\otimes K_X$ for all $1\, \leq\, i\, \leq\, r-1$ (see the 
fourth statement in Definition \ref{def1}), while the second statement follows from the 
fact that the second fundamental form of the subbundle $F_i\, \subset\, V$, for the 
holomorphic connection $D$, is nowhere zero on $X_0$ (see the fifth statement in Definition 
\ref{def1}).
\end{proof}

For any $1\, \leq\, i\, \leq\, r-1$, let
$$
{\rm SF}(\mathcal{D},\, \widehat{H}_i)
\, \in\, H^0(X,\, \text{Hom}(\widehat{H}_i,\, \widehat{H}_r/\widehat{H}_i)\otimes K_X\otimes
{\mathcal O}_X(S))
$$
be the second fundamental form of the subbundle
$\widehat{H}_i$ in \eqref{e23} for the logarithmic connection $\mathcal D$
on $J^{r-1}(Q)\otimes {\mathcal O}_X(-(r-1)S)\,=\, \widehat{H}_r$
constructed in Proposition \ref{prop3}; see \eqref{e2l} for the
second fundamental form. From Corollary
\ref{cor4}(1) we know that ${\rm SF}(\mathcal{D},\, \widehat{H}_i)$ lies in the
image of the natural inclusion map
$$
H^0(X,\, \text{Hom}(\widehat{H}_i,\, \widehat{H}_{i+1}/\widehat{H}_i)\otimes K_X\otimes
{\mathcal O}_X(S))\, \hookrightarrow\,
H^0(X,\, \text{Hom}(\widehat{H}_i,\, \widehat{H}_r/\widehat{H}_i)\otimes K_X\otimes
{\mathcal O}_X(S))
$$
given by the inclusion map $\widehat{H}_{i+1}\, \hookrightarrow\,\widehat{H}_r$. So
we get a homomorphism
\begin{equation}\label{x1}
{\rm SF}(\mathcal{D},\, i)\, \in\, H^0(X,\, \text{Hom}(\widehat{H}_i/\widehat{H}_{i-1},\,
\widehat{H}_r/\widehat{H}_i)\otimes K_X\otimes
{\mathcal O}_X(S))
\end{equation}
for every $1\, \leq\, i\, \leq\, r-1$ given by these ${\rm SF}(\mathcal{D},\, \widehat{H}_i)$.

For the filtration of $J^{r-1}(Q)\otimes {\mathcal O}_X(-(r-1)S)$, from \eqref{j1} and \eqref{e10a}
it follows immediately that
\begin{equation}\label{eq}
\widehat{H}_j/\widehat{H}_{j-1}\,=\, Q\otimes K^{\otimes (r-j)}_X\otimes{\mathcal O}_X(-(r-1)S)
\end{equation}
for all $1\, \leq\, j\, \leq\, r$. Therefore, ${\rm SF}(\mathcal{D},\, i)$ in \eqref{x1}
is a holomorphic section
\begin{equation}\label{x2}
{\rm SF}(\mathcal{D},\, i)\, \in\, H^0(X,\, {\mathcal O}_X(S))
\end{equation}
for every $1\, \leq\, i\, \leq\, r-1$.

The following lemma is a refinement of Corollary \ref{cor4}(2).

\begin{lemma}\label{lem-sf}
The holomorphic section ${\rm SF}(\mathcal{D},\, i)$ in \eqref{x2} coincides
with the section of ${\mathcal O}_X(S)$ given by the constant function $1$ on $X$.
\end{lemma}

\begin{proof}
In the proof of Corollary \ref{cor4} it was observed that
$$
(J^{r-1}(Q)\otimes {\mathcal O}_X(-(r-1)S),\, \{\widehat{H}_i\}_{i=1}^r,\,
{\mathcal D})\big\vert_{X_0}\,=\, (V,\, \{F_i\}_{i=1}^r,\, D)\big\vert_{X_0}\, ,
$$
where $X_0$ the open subset $X_0$ in \eqref{e11}. Therefore, the lemma follows from
the fifth statement in Definition \ref{def1} and Remark \ref{rem1}.
\end{proof}

Using \eqref{je} it is deduced that
$$
\det J^{r-1}(Q)\, :=\, \bigwedge\nolimits^{r}J^{r-1}(Q)\,=\,
Q^{\otimes r}\otimes K^{\otimes r(r-1)/2}_X\, .
$$
Therefore, from \eqref{e7} it follows that
$$
\det J^{r-1}(Q)\,=\, F^{\otimes r}_1\otimes (TX)^{\otimes r(r-1)/2}\otimes {\mathcal O}_X(r(r-1)S)\, .
$$
Now the expression of $F_1$ in the first statement in Definition \ref{def1} gives that
$$
\det J^{r-1}(Q)\,=\, {\mathcal O}_X\left(\frac{r(r-1)}{2}S\right)\, .
$$
Hence we have the following:
\begin{equation}\label{e27}
\det (J^{r-1}(Q)\otimes{\mathcal O}_X(-(r-1)S))\,=\,
{\mathcal O}_X\left(-\frac{r(r-1)}{2}S\right)\, .
\end{equation}

Take any branched $\text{SL}(r,{\mathbb C})$--oper $(V,\, {\mathcal F},\, D)$ as in \eqref{gbo}.
Let $\mathcal D$ denote the logarithmic connection on $J^{r-1}(Q)\otimes {\mathcal
O}_X(-(r-1)S)$ constructed in Proposition \ref{prop3} from $(V,\, {\mathcal F},\, D)$.

\begin{lemma}\label{lem4}
The logarithmic connection on $\det (J^{r-1}(Q)\otimes{\mathcal O}_X(-(r-1)S))$ induced by the
logarithmic connection $\mathcal D$ on $J^{r-1}(Q)\otimes {\mathcal
O}_X(-(r-1)S)$ coincides with the logarithmic connection on ${\mathcal O}_X\left(-\frac{r(r-1)}{2}S\right)$
given by the de Rham differential $d$ (it sends $f$ to $df$, in particular, constant
functions are covariant constant), once $\det (J^{r-1}(Q)\otimes{\mathcal O}_X(-(r-1)S))$ is
identified with ${\mathcal O}_X\left(-\frac{r(r-1)}{2}S\right)$ using \eqref{e27}.
\end{lemma}

\begin{proof}
Consider the homomorphism $\mathbf i$ in Corollary \ref{cor3}. Let
$$\wedge^r {\mathbf i}\, :\, \det (J^{r-1}(Q)\otimes {\mathcal O}_X(-(r-1)S))\,=\,
{\mathcal O}_X\left(-\frac{r(r-1)}{2}S\right) \, \longrightarrow\, \det V\,=\, {\mathcal O}_X$$
be the corresponding homomorphism of the top exterior products. This homomorphism
$\bigwedge^r {\mathbf i}$ evidently coincides with the natural inclusion map of ${\mathcal O}_X
\left(-\frac{r(r-1)}{2}S\right)$ into ${\mathcal O}_X$. Now, since $D$ induces the connection
on $\det V$ given by the de Rham differential (see the third
statement in Definition \ref{def1}), and $D$ restricts to $\mathcal D$ on 
${\mathbf i}(J^{r-1}(Q)\otimes {\mathcal O}_X(-(r-1)S))$, the lemma follows.
\end{proof}

In Section \ref{sec6} we will determine all the logarithmic connections on the rank $r$ vector bundle
$J^{r-1}(Q)\otimes {\mathcal O}_X(-(r-1)S)$ that arise from branched $\text{SL}(r,{\mathbb C})$--opers
on $X$ with branching over $S$.

\section{Residues}

Take a point $y\, \in\, 
S$. The fiber of $K_X\otimes{\mathcal O}_X(S)$ over $y$ is identified with $\mathbb C$ by 
the Poincar\'e adjunction formula \cite[p.~146]{GH}. To explain this isomorphism
\begin{equation}\label{pa}
(K_X\otimes{\mathcal O}_X(S))_y \, \stackrel{\sim}{\longrightarrow}\, {\mathbb C}\, ,
\end{equation}
let $z$ be a holomorphic coordinate function on $X$ defined on an analytic open 
neighborhood of $y$ such that $z(y)\,=\, 0$. Then we have an isomorphism ${\mathbb C}\, 
\longrightarrow\, (K_X\otimes{\mathcal O}_X(S))_y$ that sends any $c\, \in\, \mathbb C$ to 
$c\cdot \frac{dz}{z}(y)\,\in\, (K_X\otimes{\mathcal O}_X(S))_y$. It is straightforward to 
check that this map ${\mathbb C}\, \longrightarrow\, (K_X\otimes{\mathcal O}_X(S))_y$ is 
independent of the choice of the holomorphic coordinate function $z$.

Let $D_W\, :\, W\, \longrightarrow\, W\otimes K_X\otimes{\mathcal O}_X(S)$ be a 
logarithmic connection on a holomorphic vector bundle $W$ on $X$. Consider the 
composition of homomorphisms
$$
W\, \xrightarrow{\,\ D_W\,\ }\, W\otimes K_X\otimes{\mathcal O}_X(S) \, \longrightarrow\,
(W\otimes K_X\otimes{\mathcal O}_X(S))_y\,=\, W_y\,.
$$
This composition of homomorphisms is evidently ${\mathcal O}_X$--linear and therefore 
is given by an endomorphism
$$
{\rm Res}(D_W,\,y)\, :\, W_y\, \longrightarrow\, W_y\, ;
$$
 called the \textit{residue} of $D_W$ at $y$.

If $r\,=\, \text{rank}(W)$ and $\lambda_1,\, \cdots,\, \lambda_r$ are the generalized eigenvalues, with multiplicity,
of ${\rm Res}(D_W,\,y)$, then the generalized eigenvalues of the local monodromy of $D_W$ around $y$ are
$$
\exp(-2\pi\sqrt{-1}\lambda_1),\, \exp(-2\pi\sqrt{-1}\lambda_1),\, \cdots,\, \exp(-2\pi\sqrt{-1}\lambda_r);
$$
see \cite{De}.

For convenience we use the notation
\begin{equation}\label{ej}
E\, :=\, \widehat{H}_{r} \, =\, J^{r-1}(Q)\otimes{\mathcal O}_X(-(r-1)S)
\end{equation}
(see \eqref{e23}).

Take a branched $\text{SL}(r,{\mathbb C})$--oper $(V,\, {\mathcal F},\, D)$ as in \eqref{gbo}.
As before, let $\mathcal D$ denote the logarithmic connection on $E$ (see \eqref{ej})
constructed in Proposition \ref{prop3} from $(V,\, {\mathcal F},\, D)$. Recall from
Remark \ref{rem3} that the singular locus of $\mathcal D$ is exactly $S$. The
following lemma describes the residues of $\mathcal D$.

\begin{lemma}\label{lem2}
For any point $x'\, \in\, S$, let
$$
{\rm Res}({\mathcal D},\, x')\, \in\, {\rm End}(E_{x'})
$$
be the residue of ${\mathcal D}$ at $x'$. Then the eigenvalues of ${\rm Res}({\mathcal D},\, x')$ are the
integers $\{0,\, 1,\, \cdots,\, r-2,\, r-1\}$, and the multiplicity of each of them is one. For any
$0\, \leq\, i\, \leq\, r-1$, the eigenspace of ${\rm Res}({\mathcal D},\, x')$ for the eigenvalue $i$ is
contained in the subspace
$$
(\widehat{H}_{i+1})_{x'}\, \subset\, E_{x'}
$$
(see \eqref{e23} and \eqref{ej}).
\end{lemma}

\begin{proof}
Let $L$ be a holomorphic line bundle on the open neighborhood $U\, \subset\, X$ of $x'$ (as
in \eqref{e24}) equipped with a holomorphic connection $D_L$. Then for any integer $k$,
the differential operator $D_L\big\vert_{U\setminus\{x'\}}$ on
$L\big\vert_{U\setminus\{x'\}}\,=\, (L\otimes {\mathcal O}_U(kx'))\big\vert_{U\setminus\{x'\}}$
extends to a logarithmic connection on $L\otimes {\mathcal O}_U(kx')$ over $U$. In fact it coincides
with the logarithmic connection on $L\otimes {\mathcal O}_U(kx')$ given by the holomorphic
connection $D_L$ on $L$ and the logarithmic connection on ${\mathcal O}_U(kx')$ defined
by the de Rham differential $d$. The residue, at
$x'$, of this logarithmic connection on $L\otimes {\mathcal O}_U(kx')$ is $-k$;
indeed, this follows immediately from the fact that the residue of the
logarithmic connection on ${\mathcal O}_U(kx')$ defined
by the de Rham differential $d$ is $-k$. In view of this, the
lemma follows from the expression of $D\big\vert_U$ in \eqref{e25} in terms of the direct sum
of line bundles ${\mathcal G}_i$ in \eqref{i2}. Recall the observation in the proof of Proposition
\ref{prop3} that the section $\gamma_i$ in \eqref{gi} produces a
holomorphic homomorphism
$$
{\mathcal G}_i\, \longrightarrow\, {\mathcal G}_{i+1}\otimes (K_X\big\vert_U).
$$
Therefore, $\gamma_i$ does not contribute to the residue ${\rm Res}({\mathcal D},\, x')$.
Since $\widetilde{\alpha}_{i,j}$ in \eqref{i3} does not have any pole as a homomorphism from
${\mathcal G}_j$ to ${\mathcal G}_i\otimes (K_X\big\vert_U)$, it also does not
contribute to the residue ${\rm Res}({\mathcal D},\, x')$.

Therefore, the residue of ${\mathcal D}$ at $x'$ is given by the residues on the 
logarithmic connections on ${\mathcal G}_j$, $1\, \leq\, j\, \leq\, r$, induced by the 
holomorphic connections $D_j$ on ${\mathcal F}_j$ in \eqref{e25}. From the above observation 
on the residue of the logarithmic connection on $L\otimes {\mathcal O}_U(kx')$ we know that 
the residues on this logarithmic connection ${\mathcal G}_j$ is $j-1$. This proves the 
lemma.
\end{proof}

We will end this section on residues by noting some general properties of it.

Let $W\, \longrightarrow\, X$ be a holomorphic vector bundle, and let $D_W$ be a logarithmic connection on $W$,
singular at $y\, \in\, X$. Assume that the residue ${\rm Res}(D_W,\, y)$ of
$D_W$ at $y$ is semisimple, meaning ${\rm Res}(D_W,\, y)$ is diagonalizable. Let $\lambda_1,\, \cdots,\,
\lambda_b$ be the eigenvalues of ${\rm Res}(D_W,\, y)$ (they need not be
of multiplicity one). For any $1\, \leq\, i\, \leq\, b$, let
$$
W^i_y\, \subset\, W_y
$$
be the eigenspace of ${\rm Res}(D_W,\, y)\, \in\, {\rm End}(W_y)$ for the eigenvalue
$\lambda_i$. For any given $1\,\leq\, k\, <\, b$, consider the following natural homomorphisms
$$
W \,\longrightarrow\, W_y \,\longrightarrow\, W_y\Big/\left(\bigoplus_{i=1}^k W^i_y\right)\, ;
$$
both are natural quotient maps.
The kernel of this composition of homomorphisms will be denoted by $\widetilde{W}$. So
$\widetilde{W}$ is a torsion-free coherent analytic sheaf on $X$ that fits in the following short exact sequence of
coherent analytic sheaves on $X$:
\begin{equation}\label{se}
0\, \longrightarrow\, \widetilde{W} \,\stackrel{\phi}{\longrightarrow}\, W \,\stackrel{q_1}{\longrightarrow}\,
W_y\Big/\left(\bigoplus_{i=1}^k W^i_y\right) \, \longrightarrow\,0\, .
\end{equation}

The following is straight-forward to check.

\begin{lemma}\label{lem5}
The logarithmic connection $D_W\, :\, W\, \longrightarrow\, W\otimes K_X\otimes{\mathcal O}_X(y)$
sends the subsheaf $\widetilde{W} \,\stackrel{\phi}{\hookrightarrow}\, W$ to $\widetilde{W}\otimes
K_X\otimes{\mathcal O}_X(y)$. Hence $D_W$ induces a logarithmic connection on $\widetilde{W}$.
\end{lemma}

The logarithmic connection on $\widetilde{W}$ induced by $D_W$ will be denoted by 
$\widetilde{D}$.

We will
now describe the residue of $\widetilde{D}$ at the singular point $y$.

Let
\begin{equation}\label{se2}
0\, \longrightarrow\, \text{kernel}(\phi(y))\, \longrightarrow\, \widetilde{W}_y
\, \xrightarrow{\ \phi(y)\ }\, W_y \,\stackrel{q_1}{\longrightarrow}\,
\text{cokernel}(\phi(y))\,=\, W_y\Big/\left(\bigoplus_{i=1}^k W^i_y\right) \, \longrightarrow\, 0
\end{equation}
be the exact sequence of vector spaces obtained by restricting, to the point $y$, the short exact sequence
of coherent analytic sheaves in \eqref{se}. Note that the homomorphism of fibers of vector bundles
corresponding to an injective homomorphism of coherent analytic sheaves need not be injective, so
$\text{kernel}(\phi(y))$ may be nonzero.

We will now show that there is a canonical isomorphism
\begin{equation}\label{see3}
\left(\bigoplus_{i=k+1}^b W^i_y\right)\otimes (K_X)_y\,
\stackrel{\sim}{\longrightarrow}\,\text{kernel}(\phi(y))\, .
\end{equation}

To prove \eqref{see3}, take any $w\, \in\, \left(\bigoplus_{i=k+1}^b W^i_y\right)\otimes (K_X)_y$. Using the
isomorphism $(K_X)_y\,=\, {\mathcal O}_Y(-y)_y$ (see \eqref{pa}), we have
$$
w\, \in\, \left(\bigoplus_{i=k+1}^b W^i_y\right)\otimes{\mathcal O}_Y(-y)_y\,=\,
\bigoplus_{i=k+1}^b (W^i\otimes {\mathcal O}_Y(-y))_y
\, .
$$
Now take a holomorphic section defined on an analytic open neighborhood $U\, \subset\, X$
$$
s\, \in\, H^0(U,\, (W\big\vert_U)\otimes {\mathcal O}_U(-y))
$$
such that $s(y)\,=\, w$. We note that $q_1(s)\,=\, 0$, where $q_1$ is the projection in \eqref{se}.
Therefore, from \eqref{se} it follows that $s$ is the image of a holomorphic section of $\widetilde{W}\big\vert_U$
under the homomorphism $\phi$ in \eqref{se}.
Let
$$
\widetilde{s}\, \in\, H^0(U,\, \widetilde{W}\big\vert_U)
$$
be the unique holomorphic section such that $\phi(\widetilde{s})\,=\, s$.
It can be shown that the evaluation
$$
\widetilde{s}(y)\, \in\, \widetilde{W}_y
$$
is independent of the choice of the above section $s\, \in\, H^0(U,\, (W\big\vert_U)\otimes
{\mathcal O}_U(-y))$ satisfying $s(y)\,=\, w$. Indeed, for another holomorphic section
$$
t\, \in\, H^0(U,\, (W\big\vert_U)\otimes {\mathcal O}_U(-y))
$$
with $t(y)\,=\, w$, we have
$$
s-t\, \in\, H^0(U,\, (W\big\vert_U)\otimes {\mathcal O}_U(-2y))\, ,
$$
and hence
\begin{equation}\label{se4}
\widetilde{s}-\widetilde{t}\, \in\, H^0(U,\, (\widetilde{W}\big\vert_U)\otimes {\mathcal O}_U(-y))\, ,
\end{equation}
where $\widetilde{t}\, \in\, H^0(U,\, \widetilde{W}\big\vert_U)$ is the unique section for which $\phi(\widetilde{t})
\,=\, t$. From \eqref{se4} it follows immediately that $\widetilde{s}(y)\,=\, \widetilde{t}(y)$, and
hence $\widetilde{s}(y)\, \in\, \widetilde{W}_y$ is independent of the choice of
the section $s\, \in\, H^0(U,\, (W\big\vert_U)\otimes {\mathcal O}_U(-y))$ satisfying $s(y)\,=\, w$.

For the homomorphism $\phi(y)$ in \eqref{se2} we have
$$
\phi(y)(\widetilde{s}(y))\,=\, 0\, ,
$$
because $\phi(\widetilde{s})\, =\, s\, \in\, H^0(U,\, (W\big\vert_U)\otimes {\mathcal O}_U(-y))$.
Therefore, from \eqref{se2} we conclude that
$$
\widetilde{s}(y)\, \in\, \text{kernel}(\phi(y))\, .
$$

The isomorphism in \eqref{see3} sends any $$w\, \in\, \left(\bigoplus_{i=k+1}^b 
W^i_y\right)\otimes (K_X)_y$$ to $\widetilde{s}(y)\, \in\, \text{kernel}(\phi(y))$ 
constructed above from it.

The following lemma is a straight-forward consequence of the construction of residue of a
logarithmic connection.

\begin{lemma}\label{lem6}
The residue ${\rm Res}(\widetilde{D},\, y)$ of the logarithmic connection $\widetilde{D}$
on $\widetilde{W}$ (see Lemma \ref{lem5}) has the following properties:
\begin{enumerate}
\item The eigenvalues of ${\rm Res}(\widetilde{D},\, y)$ are $\{\lambda_i\}_{i=1}^k \bigcup
\{\lambda_i+1\}_{i=k+1}^b$.

\item For any $k+1\, \leq\, i\, \leq\, b$, the eigenspace of ${\rm Res}(\widetilde{D},\, y)$ for
the eigenvalue $\lambda_i+1$ is the subspace $W^i_y\otimes (K_X)_y\, \subset\,
\widetilde{W}_y$ (see \eqref{see3} and \eqref{se2}).

\item For any $1\, \leq\, i\, \leq\, k$, the eigenspace of ${\rm Res}(\widetilde{D},\, y)$ for
the eigenvalue $\lambda_i$ is taken isomorphically to the eigenspace $W^i_y$ by the homomorphism $\phi(y)$ in
\eqref{se2}.
\end{enumerate}
\end{lemma}

\section{Local monodromy}

A logarithmic connection has local monodromy around a singular point of the connection. 
The two holomorphic vector bundles $V$ and
\begin{equation}\label{rc2}
E\,:=\, J^{r-1}(Q)\otimes{\mathcal O}_X(-(r-1)S)
\end{equation}
are identified over $X_0$ \eqref{e11} (see Corollary \ref{cor3}), and this 
identification takes the holomorphic connection ${\mathcal D}\big\vert_{X_0}$ on 
$E\big\vert_{X_0}$ in Lemma \ref{lem2} to the holomorphic connection $D\big\vert_{X_0}$ on 
$V\big\vert_{X_0}$ in \eqref{gbo}. Now $D\big\vert_{X_0}$ does not have local monodromy 
around any point $x'\,\in\, S$ because $D$ is a holomorphic connection on $V$. Hence the 
logarithmic connection ${\mathcal D}$ has trivial local monodromy around every point of 
$S$, as well. In this section we will reformulate this condition of vanishing of local 
monodromies.

Let
\begin{equation}\label{lc}
{\mathbb D}\, :\, E\, \longrightarrow\, E\otimes K_X\otimes{\mathcal O}_X(S)
\end{equation}
be a logarithmic connection on $E$ (see \eqref{rc2}) singular over $S$ that satisfies the following
four conditions:
\begin{enumerate}
\item ${\mathbb D}(\widehat{H}_i)\, \subset\, \widehat{H}_{i+1}\otimes K_X\otimes{\mathcal O}_X(S)$,
for all $1\, \leq\, i\, \leq\, r-1$, where $\{\widehat{H}_i\}_{i=0}^r$ is the filtration of
$J^{r-1}(Q)\otimes {\mathcal O}_X(-(r-1)S)$ in \eqref{e23}. This
implies that the second fundamental forms of the subbundles $\widehat{H}_i$
produce a section
$${\rm SF}(\mathbb{D},\, i)\, \in\, H^0(X,\, {\mathcal O}_X(S))$$
as in \eqref{x2} for every $1\, \leq\, i\, \leq\, r-1$.

\item For all $1\, \leq\, i\, \leq\, r-1$, the above section
${\rm SF}(\mathbb{D},\, i)\, \in\, H^0(X,\, {\mathcal O}_X(S))$ coincides with
the section of ${\mathcal O}_X(S)$ given by the constant function $1$ on $X$.

\item For every $x'\, \in\, S$, the eigenvalues of ${\rm Res}({\mathbb D},\, x')$ are the
integers $\{0,\, 1,\, \cdots,\, r-2,\, r-1\}$. Note that the multiplicity of each of
the eigenvalues is one.

\item For all $0\, \leq\, i\, \leq\, r-1$ and every $x'\, \in\, S$, the eigenspace of
${\rm Res}({\mathbb D},\, x')$ for the eigenvalue $i$ is contained in the subspace
$$
(\widehat{H}_{i+1})_{x'}\, \subset\, E_{x'}
$$
(see \eqref{e23} and \eqref{ej}).
\end{enumerate}
In other words, ${\mathbb D}$ shares all the properties of ${\mathcal D}$ stated in Lemma 
\ref{lem2}, Corollary \ref{cor4} and Lemma \ref{lem-sf}.

\begin{remark}\label{rem-ch2}
Note that we did not impose the condition that all the local monodromies of $\mathbb D$ are trivial, despite the
fact that $\mathcal D$ enjoys this property (this was shown above).
\end{remark}

Let
\begin{equation}\label{es1}
L_{x'}(i)\, \subset\, E_{x'}
\end{equation}
be the eigenline for the eigenvalue $0\, \leq\, i\, \leq\, r-1$ of the residue ${\rm Res}({\mathbb D},\, x')$
at $x'\, \in\, S$. So we have a decomposition of the fiber $E_{x'}$
\begin{equation}\label{d}
E_{x'}\,=\, \bigoplus_{i=0}^{r-1} L_{x'}(i)\, ,
\end{equation}
where $L_{x'}(i)$ is the subspace in \eqref{es1}.
The above condition that the eigenspace of ${\rm Res}({\mathbb D},\, x')$ for the
eigenvalue $i$ is contained in the subspace $(\widehat{H}_{i+1})_{x'}\, \subset\, E_{x'}$ implies that
$$
(\widehat{H}_{j+1})_{x'}\,=\, \bigoplus_{i=0}^{j} L_{x'}(i)
$$
for all $0\, \leq\, j\, \leq\, r-1$; so we have
\begin{equation}\label{e26}
L_{x'}(j)\,=\, (\widehat{H}_{j+1})_{x'}/(\widehat{H}_j)_{x'}\, .
\end{equation}

We recall from \eqref{eq} that for the filtration of $E$ given in \eqref{e23},
$$
\widehat{H}_j/\widehat{H}_{j-1}\,=\, Q\otimes K^{\otimes (r-j)}_X\otimes{\mathcal O}_X(-(r-1)S)
$$
for all $1\, \leq\, j\, \leq\, r$.

Take any point $x'\, \in\, S$, and also take any integer $2\, \leq\, j\, \leq\, r$. We will construct,
from ${\mathbb D}$, an element
\begin{equation}\label{mo}
M_j({\mathbb D},\, x')\, \in\, {\rm Hom}((\widehat{H}_j/\widehat{H}_{j-1})_{x'}\otimes (T_{x'}X)^{\otimes (j-1)},\,
(\widehat{H}_{j-1}/\widehat{H}_{j-2})_{x'}\otimes (T_{x'}X)^{\otimes (j-2)})
\end{equation}
$$
=\, {\rm Hom}((\widehat{H}_j/\widehat{H}_{j-1})_{x'},\,(\widehat{H}_{j-1}/\widehat{H}_{j-2})_{x'})\otimes
(K_X)_{x'} \, =\,(K^{\otimes 2}_X)_{x'}\, ,
$$
where $T_{x'}X$ is the holomorphic tangent space to $X$ at $x'$; see \eqref{eq} for
the equality $${\rm Hom}((\widehat{H}_j/\widehat{H}_{j-1})_{x'},\,
(\widehat{H}_{j-1}/\widehat{H}_{j-2})_{x'})\, =\,(K_X)_{x'}$$ in \eqref{mo}.

To construct $M_j({\mathbb D},\, x')$, first note that the fiber ${\mathcal O}_X((j-1)S)_{x'}$ is identified with
$(T_{x'}X)^{\otimes (j-1)}$ using the Poincar\'e
adjunction formula (see \eqref{pa}). This produces an isomorphism
\begin{equation}\label{p1}
\eta\, :\, (\widehat{H}_j/\widehat{H}_{j-1})_{x'}\otimes {\mathcal O}_X((j-1)S)_{x'}\,
\stackrel{\sim}{\longrightarrow}\, (\widehat{H}_j/\widehat{H}_{j-1})_{x'}\otimes (T_{x'}X)^{\otimes (j-1)}
\end{equation}

Take any
\begin{equation}\label{w}
w\, \in\, (\widehat{H}_j/\widehat{H}_{j-1})_{x'}\otimes (T_{x'}X)^{\otimes (j-1)}\, .
\end{equation}
Using \eqref{e26} and $\eta$ (constructed in \eqref{p1}) we have
$$(\widehat{H}_j/\widehat{H}_{j-1})_{x'}\otimes (T_{x'}X)^{\otimes (j-1)}\,=\, L_{x'}(j-1)\otimes
{\mathcal O}_X((j-1)x')_{x'}\, .$$ Let
$$
w'\, \in\, L_{x'}(j-1)\otimes {\mathcal O}_X((j-1)x')_{x'}
$$
be the element that corresponds to $w$ (see \eqref{w}) by this isomorphism.

Now we choose a holomorphic section
\begin{equation}\label{ww}
\widetilde{w}\, \in\, H^0\left(U,\, (E\big\vert_U)\otimes {\mathcal O}_U((j-1)x')\right)\, ,
\end{equation}
defined on some sufficiently small analytic neighborhood
$U\, \subset\, X$ of $x'$, such that $$\widetilde{w}(x')\,=\, w'\, .$$

Let $\widehat{\mathbb D}$ be the logarithmic connection on $(E\big\vert_U)\otimes {\mathcal O}_U((j-1)x')$
induced by the logarithmic connection ${\mathbb D}\big\vert_U$ on $E\big\vert_U$ and
the logarithmic connection on ${\mathcal O}_U((j-1)x')$ given by the de Rham differential $d$. Consider the
residue ${\rm Res}(\widehat{\mathbb D},\, x')\,\in\, {\rm End}(E_{x'}\otimes
{\mathcal O}_U((j-1)x')_{x'})$ of the logarithmic connection $\widehat{\mathbb D}$ at the
point $x'$. It can be shown that the line
$$
L_{x'}(j-1)\otimes {\mathcal O}_U((j-1)x')_{x'}\, \subset\, E_{x'}\otimes {\mathcal O}_U((j-1)x')_{x'}
$$
(see \eqref{d}) is contained in the eigenspace of ${\rm Res}(\widehat{\mathbb D},\, x')$ for the eigenvalue
$0$. Indeed, ${\rm Res}({\mathbb D},\, x')$ acts on $L_{x'}(j-1)$ as multiplication by $j-1$ and the
logarithmic connection on ${\mathcal O}_U((j-1)x')$ given by the de Rham differential $d$ has the property
that its residue at $x'$ is $1-j$. Hence $L_{x'}(j-1)\otimes {\mathcal O}_U((j-1)x')_{x'}$
is contained in the eigenspace of ${\rm Res}(\widehat{\mathbb D},\, x')$ for the eigenvalue $0$. Actually,
$L_{x'}(j-1)\otimes {\mathcal O}_U((j-1)x')_{x'}$ is the eigenspace of
${\rm Res}(\widehat{\mathbb D},\, x')$ for the eigenvalue $0$.

Since the section $\widetilde{w}$ in \eqref{ww} satisfies the condition
$\widetilde{w}(x')\, \in\, L_{x'}(j-1)\otimes {\mathcal O}_U((j-1)x')_{x'}$, from the above property
of ${\rm Res}(\widehat{\mathbb D},\, x')$ that $L_{x'}(j-1)\otimes {\mathcal O}_U((j-1)x')_{x'}$
is contained in the eigenspace of ${\rm Res}(\widehat{\mathbb D},\, x')$ for the eigenvalue $0$ it follows that
\begin{equation}\label{d1}
{\mathbb D}(\widetilde{w})\, \in\, H^0\big(U,\, (E\big\vert_U)\otimes K_U\otimes
{\mathcal O}_U((j-1)x')\big)\, ,
\end{equation}
where $K_U\,:=\, K_X\big\vert_U$.

The decomposition of $E_{x'}$ in \eqref{d} gives a decomposition
\begin{equation}\label{d2}
E_{x'}\otimes (K_U\otimes {\mathcal O}_U((j-1)x'))_{x'} \,=\, \bigoplus_{i=0}^{r-1} L_{x'}(i)
\otimes (K_U\otimes {\mathcal O}_U((j-1)x'))_{x'}\, .
\end{equation}
Let
\begin{equation}\label{d3}
\beta_{j-2}(w)\, \in\, L_{x'}(j-2)\otimes (K_U\otimes {\mathcal O}_U((j-1)x'))_{x'}
\end{equation}
be the component of ${\mathbb D}(\widetilde{w})(x')\, \in\, (E\otimes K_U\otimes {\mathcal O}_U((j-1)x'))_{x'}$
(see \eqref{d1}) in
$$L_{x'}(j-2)\otimes (K_U\otimes {\mathcal O}_U((j-1)x'))_{x'}\, \subset\,
E_{x'}\otimes (K_U\otimes {\mathcal O}_U((j-1)x'))_{x'}
$$
with respect to the decomposition in \eqref{d2}. Since $L_{x'}(j-2)\,=\,
({\widehat H}_{j-1})_{x'}/({\widehat H}_{j-2})_{x'}$ (see \eqref{e26}), and
${\mathcal O}_U((j-1)x'))_{x'}\,=\, (T_{x'}X)^{\otimes (j-1)}$ (see \eqref{pa}), the element
$\beta_{j-2}(w)$ in \eqref{d3} is also an element
\begin{equation}\label{d4}
\beta_{j-2}(w)\, \in\, ({\widehat H}_{j-1}/{\widehat H}_{j-2})_{x'} \otimes (T_{x'}X)^{\otimes (j-2)}.
\end{equation}

The map $M_j({\mathbb D},\, x')$ in \eqref{mo} sends the element $w$ in \eqref{w} to
$\beta_{j-2}(w)$ constructed in \eqref{d4}.

But we need to show that this map is well-defined
in the sense that $\beta_{j-2}(w)$ depends only on $w$, in other words, $\beta_{j-2}(w)$ is independent
of the choice of the section $\widetilde{w}$ in \eqref{ww}. The following lemma shows that
$\beta_{j-2}(w)$ depends only on $w$.

\begin{lemma}\label{lem3}
The element $\beta_{j-2}(w)$ constructed in \eqref{d4} does not depend on the choice
of the section $\widetilde w$ in \eqref{ww}.
\end{lemma}

\begin{proof}
We may replace $\widetilde w$ in \eqref{ww} by $\widetilde{w}+t$, where
$$
t\, \in\, H^0\big(U,\, (E\big\vert_U)\otimes {\mathcal O}_U((j-1)x')\big)
$$
with $t(x')\,=\, 0$, where $E$ is defined in \eqref{ej}. Let 
$$
\beta^t_{j-2}(w) \, \in\, ({\widehat H}_{j-1}/{\widehat H}_{j-2})_{x'} \otimes (T_{x'}X)^{\otimes (j-2)}
$$
be the element constructed as in \eqref{d4} after substituting $\widetilde{w}+t$ in place
of $\widetilde{w}$ in the construction of $\beta_{j-2}(w)$. To prove the lemma we need to show that
\begin{equation}\label{s}
\beta^t_{j-2}(w)\,=\, \beta_{j-2}(w)\, .
\end{equation}

To prove \eqref{s}, first note that
\begin{equation}\label{s0}
t\,\in\, H^0\big(U,\, (E\big\vert_U)\otimes {\mathcal O}_U((j-2)x')\big)\, \subset
H^0\big(U,\, (E\big\vert_U)\otimes {\mathcal O}_U((j-1)x')\big)
\end{equation}
because of the given condition that $t(x')\,=\, 0$.

Let $\widehat{\mathbb D}_1$ be the logarithmic connection on $(E\big\vert_U)\otimes {\mathcal O}_U((j-2)x')$
given by the logarithmic connection ${\mathbb D}\big\vert_U$ on $E\big\vert_U$
and the logarithmic connection on ${\mathcal O}_U((j-2)x')$ given by the de Rham differential $d$.
We note that $\widehat{\mathbb D}_1$ is simply the restriction of
the logarithmic connection $\widehat{\mathbb D}$ to the subsheaf
$$
(E\big\vert_U)\otimes {\mathcal O}_U((j-2)x')\, \subset\,
(E\big\vert_U)\otimes {\mathcal O}_U((j-1)x')\, .
$$
{}From \eqref{s0} we have
$$
\widehat{\mathbb D}_1(t)\, \in\, H^0\big(U,\, (E\big\vert_U)\otimes K_U\otimes {\mathcal O}_U((j-1)x')\big)\, .
$$
Using the isomorphism in \eqref{pa}, the evaluation, at $x'$, of this section $\widehat{\mathbb D}_1(t)$
is considered as an element of
\begin{equation}\label{s1}
\widehat{\mathbb D}_1(t)(x')\, \in\, E_{x'}\otimes (K_U\otimes {\mathcal O}_U((j-1)x'))_{x'}\,
=\, E_{x'}\otimes {\mathcal O}_U((j-2)x')_{x'}\, .
\end{equation}

The decomposition in \eqref{d} produces a decomposition
\begin{equation}\label{s2}
E_{x'}\otimes {\mathcal O}_U((j-2)x')_{x'}\,=\,
\bigoplus_{i=0}^{r-1} L_{x'}(i) \otimes{\mathcal O}_U((j-2)x')_{x'}\, .
\end{equation}
The residue of the logarithmic connection $\widehat{\mathbb D}_1$ at $x'$
$$
{\rm Res}(\widehat{\mathbb D}_1,\, x')\, \in\, \text{End}(E_{x'}\otimes {\mathcal O}_U((j-2)x')_{x'})
$$
preserves the decomposition in \eqref{s2}. Moreover, ${\rm Res}(\widehat{\mathbb D}_1,\, x')$ acts
on the subspace
$$
L_{x'}(i) \otimes{\mathcal O}_U((j-2)x')_{x'}\, \subset\, E_{x'}\otimes {\mathcal O}_U((j-2)x')_{x'}
$$
in \eqref{s2} as multiplication by $i-j+2$. Indeed, the residue ${\rm Res}({\mathbb D},\, x')$ acts
on $L_{x'}(i)$ as multiplication by $i$ (see \eqref{es1}), and the residue, at $x'$, of the logarithmic
connection on ${\mathcal O}_U((j-2)x')$ given by the de Rham differential $d$ is $2-j$. Consequently,
${\rm Res}(\widehat{\mathbb D}_1,\, x')$ acts on $L_{x'}(i) \otimes{\mathcal O}_U((j-2)x')_{x'}$
as multiplication by $i-j+2$. This implies that
\begin{equation}\label{s3}
{\rm Res}(\widehat{\mathbb D}_1,\, x')(E_{x'}\otimes {\mathcal O}_U((j-2)x')_{x'})\, \subset\,
\bigoplus_{i\in\{0,\cdots, r-1\}\setminus\{j-2\}} L_{x'}(i) \otimes{\mathcal O}_U((j-2)x')_{x'}\, ,
\end{equation}
and $\text{kernel}({\rm Res}(\widehat{\mathbb D}_1,\, x'))\,=\, L_{x'}(j-2) \otimes{\mathcal O}_U((j-2)x')_{x'}$.

On the other hand, the evaluation $\widehat{\mathbb D}_1(t)(x')\, \in\,
E_{x'}\otimes {\mathcal O}_U((j-2)x')_{x'}$ in \eqref{s1} satisfies the identity
$$
\widehat{\mathbb D}_1(t)(x')\, =\, {\rm Res}(\widehat{\mathbb D}_1,\, x')\big((t)(x')\big)\, .
$$
Therefore, from \eqref{s3} it follows that
\begin{equation}\label{np}
\widehat{\mathbb D}_1(t)(x')\, \in \, \bigoplus_{i\in\{0,\cdots, r-1\}\setminus\{j-2\}}
L_{x'}(i) \otimes{\mathcal O}_U((j-2)x')_{x'}\, .
\end{equation}

Recall that $\beta_{j-2}(w)$ in \eqref{d3} is the component of ${\mathbb D}(\widetilde{w})(x')$ in
$$
L_{x'}(j-2)\otimes (K_U\otimes {\mathcal O}_U((j-1)x'))_{x'}\, =\,
L_{x'}(j-2)\otimes {\mathcal O}_U((j-2)x')_{x'}
$$
with respect to the decomposition in \eqref{d2}. Therefore, from
\eqref{np} it follows immediately that the equality in \eqref{s} holds. As noted before,
\eqref{s} completes the proof.
\end{proof}

The map $M_j({\mathbb D},\, x')$ in \eqref{mo} is defined by sending any element $w$ as in \eqref{w} to
$\beta_{j-2}(w)$ constructed in \eqref{d4} from $w$. Lemma \ref{lem3} ensures that it is well-defined.

\begin{remark}\label{rem4}
Let $D'$ be a logarithmic connection singular at a point $x$, such that residue
at $x$ is semisimple. If $\lambda$ is an
eigenvalue of the local monodromy of $D'$ around $x$, then $\lambda\,=\, \exp(2\pi\sqrt{-1}b)$,
where $b$ is an eigenvalue of the residue $\text{Res}(D',\, x)$ \cite{De}. For any point $x'\,\in\, S$, the
eigenvalues of the residue $\text{Res}({\mathbb D},\, x')$ of the connection ${\mathbb D}$ in \eqref{lc}
are integers. Therefore, we conclude that $1$ is the only eigenvalue of the local monodromy of ${\mathbb D}$
around the point $x'$. In other words, the local monodromy of ${\mathbb D}$
around $x'$ is a unipotent automorphism. This local monodromy is given by
$$
(M_2({\mathbb D},\, x'),\, M_3({\mathbb D},\, x'),\, \cdots,\, M_r({\mathbb D},\, x'))\, \in\,
((K^{\otimes 2}_X)_{x'})^{\oplus (r-1)}\, ,
$$
where the elements $M_j({\mathbb D},\, x')$ are constructed in \eqref{mo}. The elements $M_j({\mathbb D},\, x')$ will
be studied in the next section.
\end{remark}

\section{Characterizing the logarithmic connections}\label{sec6}

As in \eqref{lc}, let
\begin{equation}\label{ls}
{\mathbb D}\, :\, E\,:=\, J^{r-1}(Q)\otimes{\mathcal O}_X(-(r-1)S)\, \longrightarrow\,
J^{r-1}(Q)\otimes{\mathcal O}_X(-(r-2)S)\otimes K_X
\end{equation}
satisfying the following four conditions:
\begin{enumerate}
\item ${\mathbb D}(\widehat{H}_i)\, \subset\, \widehat{H}_{i+1}\otimes K_X\otimes{\mathcal 
O}_X(S)$, for all $1\, \leq\, i\, \leq\, r-1$, where $\{\widehat{H}_i\}_{i=0}^r$ is the 
filtration of $J^{r-1}(Q)\otimes {\mathcal O}_X(-(r-1)S)$ in \eqref{e23}. This implies that 
the second fundamental forms of the subbundles $\widehat{H}_i$ produce a section
\begin{equation}\label{x3}
{\rm SF}(\mathbb{D},\, i)\, \in\, H^0(X,\, {\mathcal O}_X(S))
\end{equation}
as in \eqref{x2} for every $1\, \leq\, i\, \leq\, r-1$.

\item For all $1\, \leq\, i\, \leq\, r-1$, the section
${\rm SF}(\mathbb{D},\, i)$ in \eqref{x3} coincides with
the section of ${\mathcal O}_X(S)$ given by the constant function $1$ on $X$.

\item For every $x'\, \in\, S$, the eigenvalues of ${\rm Res}({\mathbb D},\, x')$ are the 
integers
\begin{equation}\label{ele}
\{0,\, 1,\, \cdots,\, r-2,\, r-1\},
\end{equation}
with the multiplicity of each of them being one.

\item For all $0\, \leq\, i\, \leq\, r-1$ and every $x'\, \in\, S$, the eigenspace of ${\rm 
Res}({\mathbb D},\, x')$ for the eigenvalue $i$ is contained in the subspace 
$(\widehat{H}_{i+1})_{x'}\, \subset\, E_{x'}$ in \eqref{ej}.
\end{enumerate}

\begin{theorem}\label{thm1}
There is a branched ${\rm SL}(r,{\mathbb C})$--oper $$(V,\, {\mathcal F},\, D)$$ such that
the logarithmic connection ${\mathbb D}$ in \eqref{ls}
coincides with the logarithmic connection on $E$ associated to $(V,\, {\mathcal F},\, D)$ by
Proposition \ref{prop3} if and only if the following two conditions hold:
\begin{enumerate}
\item The logarithmic connection on $\det (J^{r-1}(Q)\otimes{\mathcal O}_X(-(r-1)S))$ induced by the
logarithmic connection $\mathbb D$ on $J^{r-1}(Q)\otimes {\mathcal
O}_X(-(r-1)S)$ coincides with the logarithmic connection on ${\mathcal O}_X\left(-\frac{r(r-1)}{2}S\right)$
given by the de Rham differential $d$, once $\det (J^{r-1}(Q)\otimes{\mathcal O}_X(-(r-1)S))$ is
identified with ${\mathcal O}_X\left(-\frac{r(r-1)}{2}S\right)$ using \eqref{e27}.

\item $M_j({\mathbb D},\, x')\,=\, 0$, for all $2\, \leq\, j\,\leq\, r$ and every $x'\, \in\, S$,
where $M_j({\mathbb D}, \,x')$ are constructed in \eqref{mo}.
\end{enumerate}
\end{theorem}

\begin{proof}
If there is a branched $\text{SL}(r,{\mathbb C})$--oper $(V,\, {\mathcal F},\, D)$ such that
${\mathbb D}$ coincides with the logarithmic connection on $J^{r-1}(Q)\otimes{\mathcal O}_X(-(r-1)S)$
associated to $(V,\, {\mathcal F},\, D)$ by Proposition \ref{prop3}, then from Lemma \ref{lem4}
we know that the logarithmic connection on $\det (J^{r-1}(Q)\otimes{\mathcal O}_X(-(r-1)S))$ induced by
$\mathbb D$ on $J^{r-1}(Q)\otimes {\mathcal
O}_X(-(r-1)S)$ coincides with the logarithmic connection on ${\mathcal O}_X\left(-\frac{r(r-1)}{2}S\right)$
given by the de Rham differential $d$, once $\det (J^{r-1}(Q)\otimes{\mathcal O}_X(-(r-1)S))$ is
identified with ${\mathcal O}_X\left(-\frac{r(r-1)}{2}S\right)$ using \eqref{e27}.

Therefore, we assume that the logarithmic connection on $\det (J^{r-1}(Q)\otimes{\mathcal O}_X(-(r-1)S))$
induced by $\mathbb D$ on $J^{r-1}(Q)\otimes {\mathcal
O}_X(-(r-1)S)$ coincides with the logarithmic connection on ${\mathcal O}_X\left(-\frac{r(r-1)}{2}S\right)$
given by the de Rham differential $d$, after $\det (J^{r-1}(Q)\otimes{\mathcal O}_X(-(r-1)S))$ is
identified with ${\mathcal O}_X\left(-\frac{r(r-1)}{2}S\right)$ using \eqref{e27}.

To prove the theorem we
need to show the following: There is a branched $\text{SL}(r,{\mathbb C})$--oper $(V,\, {\mathcal F},\, D)$ such that
${\mathbb D}$ coincides with the logarithmic connection on $E\,=\, J^{r-1}(Q)\otimes{\mathcal O}_X(-(r-1)S)$
associated to $(V,\, {\mathcal F},\, D)$ by Proposition \ref{prop3} if and only if
$M_j({\mathbb D},\, x')\,=\, 0$ for all $2\, \leq\, j\,\leq\, r$.

In Proposition \ref{prop3} we constructed a logarithmic connection on $E$ from a
branched $\text{SL}(r,{\mathbb C})$--oper. The above statement will be proved by establishing an
inverse of this construction in Proposition \ref{prop3}.

Consider the holomorphic vector bundle
\begin{equation}\label{e28}
{\mathcal W}\, :=\, J^{r-1}(Q)\,=\, E\otimes {\mathcal O}_X((r-1)S)
\end{equation}
on $X$ (see \eqref{ej}). The logarithmic connection $\mathbb D$ on $E$ and the logarithmic connection on
${\mathcal O}_X((r-1)S)$ given by the de Rham differential $d$ together produce a logarithmic connection 
\begin{equation}\label{e28b}
\widetilde{\mathbb D}\, :\, {\mathcal W}\, \longrightarrow\, {\mathcal W}\otimes K_X\otimes {\mathcal O}_X(S)
\end{equation}
on the holomorphic vector bundle ${\mathcal W}$ in \eqref{e28}. 
At any point $x'\, \in\, S$, the residue
of the logarithmic connection on ${\mathcal O}_X((r-1)S)$ given by the de Rham differential $d$
is $1-r$. On the other hand, the eigenvalues of ${\rm Res}({\mathbb D},\, x')$ are
given to be $\{0,\, 1,\, \cdots,\, r-2,\, r-1\}$ (see \eqref{ele}). Therefore, the eigenvalues
of the residue $\text{Res}(\widetilde{\mathbb D},\, x')$ of $\widetilde{\mathbb D}$ at $x'$ are
$\{1-r,\, 2-r,\, \cdots ,\, -1,\, 0\}$. We note that the multiplicity of every eigenvalue
of $\text{Res}(\widetilde{\mathbb D},\, x')$ is one.

For each $x'\,\in\, S$, let
$$
\ell_{x'}(0)\, \subset\, {\mathcal W}_{x'}
$$
be the eigenspace, for the eigenvalue $0$, of $\text{Res}(\widetilde{\mathbb D},\, x')$; so
$\ell_{x'}(0)$ is a line in ${\mathcal W}_{x'}$. Let ${\mathcal W}_1$ be the holomorphic vector
bundle of rank $r$ on $X$ defined by the following short exact sequence of coherent analytic sheaves on $X$:
\begin{equation}\label{w1}
0\, \longrightarrow\, {\mathcal W}_1\, \longrightarrow\, {\mathcal W}\, \longrightarrow\,
\bigoplus_{x'\in S}{\mathcal W}_{x'}/\ell_{x'}(0) \, \longrightarrow\, 0\, .
\end{equation}
{}From Lemma \ref{lem5} we know that the logarithmic connection
${\mathbb D}\, :\, {\mathcal W}\, \longrightarrow\, {\mathcal W}\otimes K_X\otimes{\mathcal O}_X(S)$ preserves
the subsheaf ${\mathcal W}_1$ in \eqref{w1}. Let
$$
{\mathbb D}_1\, :\, {\mathcal W}_1\, \longrightarrow\, {\mathcal W}_1\otimes K_X\otimes{\mathcal O}_X(S)
$$
be the logarithmic connection on ${\mathcal W}_1$ induced by ${\mathbb D}$. Since the
eigenvalues of ${\rm Res}({\mathbb D},\, x')$ are $\{0,\, -1,\, \cdots,\, 2-r,\, 1-r\}$, and $\ell_{x'}(0)$
is the eigenspace of ${\rm Res}({\mathbb D},\, x')$ for the eigenvalue $0$, from Lemma \ref{lem6} it follows
that the eigenvalues of ${\rm Res}({\mathbb D}_1,\, x')$ are
$\{0,\, -1,\, \cdots,\, 2-r\}$. The eigenvalue $0$ of ${\rm Res}({\mathbb D}_1,\, x')$ has multiplicity two,
while the rest of the eigenvalues of ${\rm Res}({\mathbb D}_1,\, x')$ are of multiplicity one.

For each $x'\,\in\, S$, let
$$
H_1(x')\, \subset\, ({\mathcal W}_1)_{x'}
$$
be the eigenspace of ${\rm Res}({\mathbb D}_1,\, x')\, \in\, {\rm End}(({\mathcal 
W}_1)_{x'})$ for the eigenvalue $0$. As noted above, we have $\dim H_1(x')\,=\, 2$. 
Imitating \eqref{w1} we define ${\mathcal W}_2$. More precisely, let ${\mathcal W}_2$ be 
the holomorphic vector bundle of rank $r$ on $X$ defined by the following short exact 
sequence of coherent analytic sheaves on $X$:
\begin{equation}\label{e31}
0\, \longrightarrow\, {\mathcal W}_2\, \longrightarrow\, {\mathcal W}_1\, \longrightarrow\,
\bigoplus_{x'\in S}({\mathcal W}_1)_{x'}/H_1(x') \, \longrightarrow\, 0\, .
\end{equation}
{}From Lemma \ref{lem5} we know that the logarithmic connection
${\mathbb D}_1\, :\, {\mathcal W}_1\, \longrightarrow\, {\mathcal W}_1\otimes K_X\otimes{\mathcal O}_X(S)$ preserves
the subsheaf ${\mathcal W}_2$ \eqref{e31}. Let
$$
{\mathbb D}_2\, :\, {\mathcal W}_2\, \longrightarrow\, {\mathcal W}_2\otimes K_X\otimes{\mathcal O}_X(S)
$$
be the logarithmic connection on ${\mathcal W}_2$ induced by ${\mathbb D}_1$. Since the
eigenvalues of ${\rm Res}({\mathbb D}_1,\, x')$ are $\{0,\, -1,\, \cdots,\, 2-r\}$, from Lemma \ref{lem6} it follows
that the eigenvalues of ${\rm Res}({\mathbb D}_2,\, x')$ are
$\{0,\, -1,\, \cdots,\, 3-r\}$. The eigenvalue $0$ of ${\rm Res}({\mathbb D}_2,\, x')$ has multiplicity three.

For each $x'\,\in\, S$, let $H_2(x')\, \subset\, ({\mathcal W}_2)_{x'}$
be the eigenspace of ${\rm Res}({\mathbb D}_2,\, x')$ for the eigenvalue $0$. Define the holomorphic vector
bundle ${\mathcal W}_3$ by the short exact sequence of coherent analytic sheaves
$$
0\, \longrightarrow\, {\mathcal W}_3\, \longrightarrow\, {\mathcal W}_2\, \longrightarrow\,
\bigoplus_{x'\in S}({\mathcal W}_2)_{x'}/H_2(x') \, \longrightarrow\, 0\, .
$$
We now proceed inductively. To explain this, for $2\,\leq\, j\, \leq\, r-2$, suppose that we have 
constructed a holomorphic vector bundle ${\mathcal W}_j$, and a logarithmic connection
${\mathbb D}_j$ on it, such that the following conditions hold:
\begin{itemize}
\item For each $x'\,\in\, S$, the eigenvalues of ${\rm Res}({\mathbb D}_j,\, x')$ are
$\{0,\, -1,\, \cdots,\, j+1-r\}$.

\item The multiplicity of the eigenvalue zero of ${\rm Res}({\mathbb D}_j,\, x')$
is $j+1$.
\end{itemize}
Let
$$
H_j(x')\, \subset\, ({\mathcal W}_j)_{x'}
$$
be the eigenspace of ${\rm Res}({\mathbb D}_j,\, x')$ for the eigenvalue $0$. Then define
the holomorphic vector bundle ${\mathcal W}_{j+1}$ by the short exact sequence of coherent analytic sheaves
\begin{equation}\label{e29}
0\, \longrightarrow\, {\mathcal W}_{j+1}\, \longrightarrow\, {\mathcal W}_j\, \longrightarrow\,
\bigoplus_{x'\in S}({\mathcal W}_j)_{x'}/H_j(x') \, \longrightarrow\, 0\, .
\end{equation}
{}From Lemma \ref{lem5} we know that the logarithmic connection ${\mathbb D}_j$ on 
${\mathcal W}_j$ preserves the subsheaf ${\mathcal W}_{j+1}\,\subset\, {\mathcal W}_j$ in 
\eqref{e29}; the logarithmic connection on ${\mathcal W}_{j+1}$ induced by ${\mathbb D}_j$ 
is denoted by ${\mathbb D}_{j+1}$. Since the eigenvalues of ${\rm Res}({\mathbb D}_j,\, 
x')$ are $\{0,\, -1,\, \cdots,\, j+1-r\}$, from Lemma \ref{lem6} it follows that the 
eigenvalues of ${\rm Res}({\mathbb D}_{j+1},\, x')$ are $\{0,\, -1,\, \cdots,\, j+2-r\}$. 
The eigenvalue $0$ of ${\rm Res}({\mathbb D}_{j+1},\, x')$ has multiplicity $j+2$.

Proceeding inductively, we finally obtain the following:
\begin{enumerate}
\item a holomorphic vector bundle ${\mathcal W}_{r-1}$ on $X$ of rank $r$, and

\item a logarithmic connection ${\mathbb D}_{r-1}$ on ${\mathcal W}_{r-1}$ whose singular locus is contained in $S$,
and for each point $x'\, \in\, S$, the residue
\begin{equation}\label{rr}
{\rm Res}({\mathbb D}_{r-1},\, x')\,\in\, {\rm End}(({\mathcal W}_{r-1})_{x'})
\end{equation}
is nilpotent (meaning, zero is the only eigenvalue of it).
\end{enumerate}

The next step in the proof of the theorem is to prove the following proposition.

\begin{proposition}\label{prop4}
Take any point $x'\, \in\, S$. Then
$$
{\rm Res}({\mathbb D}_{r-1},\, x')\,=\, 0
$$
(see \eqref{rr}) if and only if $M_j({\mathbb D},\, x')\,=\, 0$, for all $2\, \leq\, j\,\leq\, r$,
where $M_j({\mathbb D},\, x')$ are constructed in \eqref{mo}.
\end{proposition}

\begin{proof}
Consider the holomorphic vector bundle $\mathcal W$ in \eqref{e28}. Let
\begin{equation}\label{vp}
\varphi\, :\, {\mathcal W}_{r-1}\, \longrightarrow\, {\mathcal W}
\end{equation}
be the following composition of homomorphisms
$$
{\mathcal W}_{r-1}\, \longrightarrow\, {\mathcal W}_{r-2} \, \longrightarrow\, \cdots \, \longrightarrow\,
{\mathcal W}_{2} \, \longrightarrow\, {\mathcal W}_{1} \, \longrightarrow\, {\mathcal W}\, ;
$$
see \eqref{e29}, \eqref{e31} and \eqref{w1} for the above homomorphisms. We note that
the homomorphism $\varphi$ in \eqref{vp} is an isomorphism over
the open subset $X_0$ in \eqref{e11}. We will now show that any
holomorphic subbundle ${\mathcal V}\, \subset\, {\mathcal W}$ produces a holomorphic subbundle of
${\mathcal W}_{r-1}$. To prove this, let
$$
{\mathcal V}'\, \subset\, {\mathcal W}_{r-1}
$$
be the coherent analytic subsheaf uniquely defined by the following condition:
A holomorphic section $\sigma\, \in\, H^0(U,\, {\mathcal W}_{r-1})$ over some analytic open subset
$U\, \subset\, X$ is a section of ${\mathcal V}'$ if and only if the restriction of $\varphi(\sigma)$ to the
complement $U\setminus (U\bigcap S)$ is a section of the subbundle ${\mathcal V}\, \subset\, {\mathcal W}$. It is
straight-forward to check that ${\mathcal V}'$ is a holomorphic subbundle of ${\mathcal W}_{r-1}$.

For $0\, \leq\, j\,\leq\, r$, consider the holomorphic subbundle
$$
H_j\, \subset\, J^{r-1}(Q) \,=\, {\mathcal W}
$$
(see \eqref{e10} and \eqref{e28}). Let
$$
{\mathcal E}_j\, \subset\, {\mathcal W}_{r-1}
$$
be the holomorphic subbundle corresponding to $H_j$. So we have the filtration of holomorphic subbundles
\begin{equation}\label{e32}
0\,=\, {\mathcal E}_0\, \subset\, {\mathcal E}_1 \, \subset\, {\mathcal E}_2
\, \subset\, \cdots \, \subset\, {\mathcal E}_{r-1} \, \subset\, {\mathcal E}_r \,=\, {\mathcal W}_{r-1}
\end{equation}
of ${\mathcal W}_{r-1}$. Note that we have
\begin{equation}\label{vp2}
\varphi({\mathcal E}_j)\, \subset\, H_j
\end{equation}
for all $0\, \leq\, j\, \leq\, r$, where $\varphi$ is the homomorphism in \eqref{vp}. Therefore,
$\varphi$ produces a homomorphism
\begin{equation}\label{vpj}
\varphi_j\, :\, {\mathcal E}_j/{\mathcal E}_{j-1} \, \longrightarrow\, H_j/H_{j-1}\,=\, Q\otimes K^{\otimes (r-j)}_X
\end{equation}
for all $1\, \leq\, j\, \leq\, r$; see \eqref{e10a} for the isomorphism in \eqref{vpj}.

Recall that the logarithmic connection $\mathbb D$ on $E\,=\, J^{r-1}(Q)\otimes{\mathcal O}_X(-(r-1)S)$
satisfies the following condition: For any $x'\, \in\, S$ and any $0\, \leq\, i\, \leq\, r-1$, the eigenspace of
${\rm Res}({\mathbb D},\, x')$ for the
eigenvalue $i$ is contained in the subspace $(\widehat{H}_{i+1})_{x'}\, \subset\, E_{x'}$ (see \eqref{e23} and
\eqref{ej}). This condition implies that the logarithmic connection $\widetilde{\mathbb D}$ on ${\mathcal W}$
in \eqref{e28b} has the following property:
For any $x'\, \in\, S$ and any $0\, \leq\, i\, \leq\, r-1$, the eigenspace of
${\rm Res}(\widetilde{\mathbb D},\, x')$ for the eigenvalue $i-r+1$ is contained in the subspace $(H_{i+1})_{x'}\,
\subset\, {\mathcal W}_{x'}$ in \eqref{e10}. Using this and \eqref{vp2} it follows that
${\rm Res}({\mathbb D}_{r-1},\, x')$ in \eqref{rr} preserves the filtration of subspaces
\begin{equation}\label{e35}
0\,=\, ({\mathcal E}_0)_{x'}\, \subset\, ({\mathcal E}_1)_{x'} \, \subset\, ({\mathcal E}_2)_{x'}
\, \subset\, \cdots \, \subset\, ({\mathcal E}_{r-1})_{x'} \, \subset\,
({\mathcal E}_r)_{x'} \,=\,({\mathcal W}_{r-1})_{x'}
\end{equation}
obtained from \eqref{e32}. Since ${\rm Res}({\mathbb D}_{r-1},\, x')$ is a nilpotent endomorphism, and
it preserves the filtration in \eqref{e35}, we conclude that
$$
{\rm Res}({\mathbb D}_{r-1},\, x')(({\mathcal E}_i)_{x'})\, \subset\, ({\mathcal E}_{i-1})_{x'}
$$
for all $1\,\leq\, i\, \leq\, r$. It also follows from the construction of ${\mathcal W}_{r-1}$ that
$$
{\rm Res}({\mathbb D}_{r-1},\, x')(({\mathcal E}_i)_{x'})\bigcap ({\mathcal E}_{i-2})_{x'}\,=\, 0
$$
for all $2\,\leq\, i\, \leq\, r$.

The above observations on ${\rm Res}({\mathbb D}_{r-1},\, x')$ combine together to give the following:

For every $x\, \in\, S$, the residue ${\rm Res}({\mathbb D}_{r-1},\, x')$ gives an element
\begin{equation}\label{e36}
\mathbb{R}({\mathbb D}_{r-1},\, x')\,\in\, \bigoplus_{i=1}^{r-1} 
{\rm Hom}({\mathcal E}_{i+1}/{\mathcal E}_i,\, {\mathcal E}_i/{\mathcal E}_{i-1})_{x'}\, .
\end{equation}
This $\mathbb{R}({\mathbb D}_{r-1},\, x')$ has the property that ${\rm Res}({\mathbb D}_{r-1},\, x')\,=\, 0$
if and only if $\mathbb{R}({\mathbb D}_{r-1},\, x')\,=\, 0$.

{}From \eqref{vpj} and the construction of ${\mathcal W}_{r-1}$ it follows that
$$
{\mathcal E}_i/{\mathcal E}_{i-1}\,=\, (H_i/H_{i-1})\otimes {\mathcal O}_X(-(r-i)S)
\,=\, Q\otimes K^{\otimes (r-i)}_X\otimes {\mathcal O}_X(-(r-i)S)
$$
for all $1\, \leq\, i\, \leq\, r$. From this we conclude that
\begin{equation}\label{e39}
{\mathcal E}_{i+1}/{\mathcal E}_i\,=\, ({\mathcal E}_i/{\mathcal E}_{i-1})\otimes TX\otimes {\mathcal O}_X(S)\, .
\end{equation}
The isomorphism in \eqref{e39} implies that for any $x'\, \in\, S$ and all $1\,\leq\, i\leq\, k-1$, we have 
$$
{\rm Hom}({\mathcal E}_{i+1}/{\mathcal E}_i,\, {\mathcal E}_i/{\mathcal E}_{i-1})_{x'}
\,=\, (K_X\otimes {\mathcal O}_X(-S))_{x'}\,=\, (K^{\otimes 2}_X)_{x'}
$$
(see \eqref{pa} for the last isomorphism). Therefore, $\mathbb{R}({\mathbb D}_{r-1}, \,x')$ in \eqref{e36}
can be considered as an element
$$
\mathbb{R}({\mathbb D}_{r-1},\, x')\,\in\, \bigoplus_{i=1}^{r-1} (K^{\otimes 2}_X)_{x'}\,=\,
((K^{\otimes 2}_X)_{x'})^{\oplus (r-1)}\, .
$$
For any $1\,\leq\, i\, \leq\, r-1$, the element of $(K^{\otimes 2}_X)_{x'}$ in the $i$--th component of
$\mathbb{R}({\mathbb D}_{r-1}, \,x')$, with respect to the above decomposition, coincides with 
$M_{i+1}({\mathbb D},\, x')$ constructed in \eqref{mo}. It was noted earlier that
${\rm Res}({\mathbb D}_{r-1},\, x')\,=\, 0$ if and only if $\mathbb{R}({\mathbb D}_{r-1},\, x')\,=\, 0$.
Therefore, the proof of the proposition is complete.
\end{proof}

Continuing with the proof of Theorem \ref{thm1}, first assume that
there is a branched $\text{SL}(r,{\mathbb C})$--oper $(V,\, {\mathcal F},\, D)$ such that
${\mathbb D}$ coincides with the logarithmic connection on $E\,=\, J^{r-1}(Q)\otimes{\mathcal O}_X(-(r-1)S)$
associated to $(V,\, {\mathcal F},\, D)$ by Proposition \ref{prop3}.

It is straight-forward to check that the construction of the triple $({\mathcal W}_{r-1},\,
\{{\mathcal E}_i\}_{i=0}^r,\, {\mathbb D}_{r-1})$ from $(J^{r-1}(Q)\otimes{\mathcal O}_X(-(r-1)S),\,
{\mathbb D})$ is the inverse of
the construction of $(J^{r-1}(Q)\otimes{\mathcal O}_X(-(r-1)S),\, {\mathbb D})$ from
$(V,\, {\mathcal F},\, D)$. More precisely, $({\mathcal W}_{r-1},\,
\{{\mathcal E}_i\}_{i=0}^r,\, {\mathbb D}_{r-1})$ coincides with $(V,\, {\mathcal F},\, D)$.
In particular, ${\mathbb D}_{r-1}$ is a holomorphic connection on ${\mathcal W}_{r-1}$, as
$D$ is a holomorphic connection on $V$. Now from Proposition \ref{prop4}
we conclude that $M_j({\mathbb D},\, x')\,=\, 0$ for all $2\, \leq\, j\,\leq\, r$ and every $x'\,\in\, S$.

To prove the converse, assume that
\begin{equation}\label{e37}
M_j({\mathbb D},\, x')\,=\, 0
\end{equation}
for all $2\, \leq\, j\,\leq\, r$ and every $x'\, \in\, S$,
where $M_j({\mathbb D},\, x')$ are constructed in \eqref{mo}. We will show that
there is a branched $\text{SL}(r,{\mathbb C})$--oper $(V,\, {\mathcal F},\, D)$ such that
${\mathbb D}$ coincides with the logarithmic connection on $E\,=\, J^{r-1}(Q)\otimes{\mathcal O}_X(-(r-1)S)$
associated to $(V,\, {\mathcal F},\, D)$ by Proposition \ref{prop3}.

Since \eqref{e37} holds, from Proposition \ref{prop4} we know that 
the logarithmic connection ${\mathbb D}_{r-1}$ on ${\mathcal W}_{r-1}$ is actually
a holomorphic connection. Consider the filtration $\{{\mathcal E}_i\}_{i=0}^r$ of ${\mathcal W}_{r-1}$
in \eqref{e32}. It can be shown that
\begin{equation}\label{e38}
{\mathbb D}_{r-1}({\mathcal E}_i) \, \subset\, {\mathcal E}_{i+1}\otimes K_X
\end{equation}
for all $0\, \leq\, i\, \leq\, r-1$. Indeed, we have
\begin{equation}\label{e41}
({\mathcal W}_{r-1},\, \{{\mathcal E}_i\}_{i=0}^r,\, {\mathbb D}_{r-1})\big\vert_{X_0}
\,=\, (J^{r-1}(Q)\otimes{\mathcal O}_X(-(r-1)S),\, \{\widehat{H}_i\}_{i=0}^r,\, {\mathbb D})
\end{equation}
over the nonempty open subset $X_0$ in \eqref{e11}; the filtration
$\{\widehat{H}_i\}_{i=0}^r$ is constructed in \eqref{e23}. So \eqref{e38} follows from the given condition that
$${\mathbb D}(\widehat{H}_i)\, \subset\, \widehat{H}_{i+1}\otimes K_X\otimes{\mathcal O}_X(S)$$
for all $0\, \leq\, i\, \leq\, r-1$.

In view of \eqref{e38}, the second fundamental form of ${\mathcal E}_i\, \subset\, {\mathcal W}_{r-1}$
for the holomorphic connection ${\mathbb D}_{r-1}$ produces a homomorphism
$$
\Psi_i\,\in\, H^0(X,\, \text{Hom}({\mathcal E}_i/{\mathcal E}_{i-1},\, {\mathcal E}_{i+1}/{\mathcal E}_i)
\otimes K_X)
$$
for every $1\, \leq\, i\, \leq\, r-1$. Now using the isomorphism in \eqref{e39} we conclude that
\begin{equation}\label{e40}
\Psi_i\,\in\, H^0(X,\, {\mathcal O}_X(S))\, .
\end{equation}

Recall the given condition that for all $1\, \leq\, i\, \leq\, r-1$, the section
${\rm SF}(\mathbb{D},\, i)$ in \eqref{x3} is given
by the constant function $1$ on $X$. Therefore, from the isomorphism in \eqref{e41} we 
conclude that the section $\Psi_i$ in \eqref{e40} coincides with the section of
${\mathcal O}_X(S)$ given by the constant function $1$ on $X$. From this
it follows that $$({\mathcal W}_{r-1},\, 
\{{\mathcal E}_i\}_{i=0}^r,\, {\mathbb D}_{r-1})$$ is a branched $\text{SL}(r,{\mathbb 
C})$--oper. The logarithmic connection on $J^{r-1}(Q)\otimes{\mathcal O}_X(-(r-1)S)$ that 
corresponds to the branched $\text{SL}(r,{\mathbb C})$--oper $({\mathcal W}_{r-1},\, 
\{{\mathcal E}_i\}_{i=0}^r,\, {\mathbb D}_{r-1})$ by Proposition \ref{prop3} coincides with 
$\mathbb D$, because the construction $({\mathcal W}_{r-1},\, \{{\mathcal E}_i\}_{i=0}^r,\, 
{\mathbb D}_{r-1})$ from $\mathbb D$ is the inverse of the construction in Proposition 
\ref{prop3}. This completes the proof of Theorem \ref{thm1}.
\end{proof}

Let
\begin{equation}\label{ls2}
{\mathbb D}\, :\, E\,:=\, J^{r-1}(Q)\otimes{\mathcal O}_X(-(r-1)S)\, \longrightarrow\,
J^{r-1}(Q)\otimes{\mathcal O}_X(-(r-2)S)\otimes K_X
\end{equation}
be a logarithmic connection satisfying the following
five conditions:
\begin{enumerate}
\item The logarithmic connection on $\det (J^{r-1}(Q)\otimes{\mathcal O}_X(-(r-1)S))$ induced by the
logarithmic connection $\mathbb D$ on $J^{r-1}(Q)\otimes {\mathcal
O}_X(-(r-1)S)$ coincides with the logarithmic connection on ${\mathcal O}_X\left(-\frac{r(r-1)}{2}S\right)$
given by the de Rham differential $d$, once $\det (J^{r-1}(Q)\otimes{\mathcal O}_X(-(r-1)S))$ is
identified with ${\mathcal O}_X\left(-\frac{r(r-1)}{2}S\right)$ using \eqref{e27}.

\item ${\mathbb D}(\widehat{H}_i)\, \subset\, \widehat{H}_{i+1}\otimes K_X\otimes{\mathcal 
O}_X(S)$, for all $1\, \leq\, i\, \leq\, r-1$, where $\{\widehat{H}_i\}_{i=0}^r$ is the 
filtration of $J^{r-1}(Q)\otimes {\mathcal O}_X(-(r-1)S)$ in \eqref{e23}. This implies that 
the second fundamental forms of the subbundles $\widehat{H}_i$ produce a section
\begin{equation}\label{x3n}
{\rm SF}(\mathbb{D},\, i)\, \in\, H^0(X,\, {\mathcal O}_X(S))
\end{equation}
as in \eqref{x2} for every $1\, \leq\, i\, \leq\, r-1$.

\item For all $1\, \leq\, i\, \leq\, r-1$, the section
${\rm SF}(\mathbb{D},\, i)$ in \eqref{x3n} coincides with
the section of ${\mathcal O}_X(S)$ given by the constant function $1$ on $X$.

\item For every $x'\, \in\, S$, the eigenvalues of ${\rm Res}({\mathbb D},\, x')$ are the 
integers $\{0,\, 1,\, \cdots,\, r-2,\, r-1\}$, with the multiplicity of each of them being 
one.

\item For all $0\, \leq\, i\, \leq\, r-1$ and every $x'\, \in\, S$, the eigenspace of ${\rm 
Res}({\mathbb D},\, x')$ for the eigenvalue $i$ is contained in the subspace 
$(\widehat{H}_{i+1})_{x'}\, \subset\, E_{x'}$ in \eqref{ej}.
\end{enumerate}

In the proof of Theorem \ref{thm1} we constructed the following:
\begin{enumerate}
\item a holomorphic vector bundle ${\mathcal W}_{r-1}$ on $X$ of rank $r$, and

\item a logarithmic connection ${\mathbb D}_{r-1}$ on ${\mathcal W}_{r-1}$ whose singular locus is contained in $S$,
and for each point $x'\, \in\, S$, the residue
$$
{\rm Res}({\mathbb D}_{r-1},\, x')\,\in\, {\rm End}(({\mathcal W}_{r-1})_{x'})
$$
is nilpotent (see \eqref{rr}).
\end{enumerate}

If the residue ${\rm Res}(\nabla,\, y)$ of a logarithmic connection $\nabla$ at a point $y$ is nilpotent, then the local monodromy
of $\nabla$ around $y$ is trivial if and only if ${\rm Res}(\nabla,\, y)\,=\, 0$. Therefore, Proposition \ref{prop4} has the
following immediate corollary:

\begin{corollary}\label{corp4}
The connection ${\mathbb D}_{r-1}$ on ${\mathcal W}_{r-1}$ is nonsingular (meaning it is a holomorphic connection)
if and only if the local monodromy of ${\mathbb D}_{r-1}$ around each point of $S$ it trivial.
\end{corollary}

In view of Corollary \ref{corp4}, from the last part of the proof of Theorem \ref{thm1} (the part after Proposition \ref{prop4})
we have the following:

\begin{proposition}\label{prop5}
There is a branched ${\rm SL}(r,{\mathbb C})$--oper $(V,\, {\mathcal F},\, D)$ such that
the logarithmic connection ${\mathbb D}$ in \eqref{ls2}
coincides with the logarithmic connection on $E$ associated to $(V,\, {\mathcal F},\, D)$ by
Proposition \ref{prop3} if and only if the local monodromy of $\mathbb D$ around each point $S$ is trivial.
\end{proposition}

\subsection{${\rm SL}(r, {\mathbb C})$-opers with regular
singularity and branched ${\rm SL}(r, {\mathbb C})$-opers}\label{se7}

This subsection was communicated to us by Edward Frenkel.

In this subsection we relate the branched ${\rm SL}(r, {\mathbb C})$-opers
with the ${\rm SL}(r, {\mathbb C})$-opers with regular singularity and
trivial monodromy introduced by Frenkel and Gaitsgory in \cite[Section 2.9]{FG}.

Let $G$ be a connected simple affine algebraic group over $\mathbb C$, and let $X$ be a 
smooth projective algebraic curve over ${\mathbb C}$. As before, $S$ is a finite subset of 
$X$. Beilinson and Drinfeld have introduced in \cite[Section 4]{BD1} and \cite[Section 
3.8]{BD2} the notion of a $G$--oper on $X$ with {\em regular singularity} at $S$. We recall 
that it is an ordinary $G$--oper on the complement of $S$ in $X$, whose restriction to the 
formal punctured disc $D^\times_{x_k}$ at each point $x_k \,\in\, S$ is a $G$--oper on the 
punctured disc with regular singularity; the details are in \cite[Section 3.8.8]{BD2} (see 
also \cite[Section 2.4]{FG}). Further, Beilinson and Drinfeld have defined the {\em 
residue} of such an oper. This residue is a point of the geometric invariant theoretic 
quotient ${\mathfrak g}/G \simeq {\mathfrak h}/W$, where ${\mathfrak g}$ is the Lie algebra 
of $G$, ${\mathfrak h}\, \subset\, {\mathfrak g}$ is its Cartan subalgebra, and $W$ is the 
corresponding Weyl group $N_G(T)/T$.

Denote by $\varpi$ the natural projection ${\mathfrak h}\,\longrightarrow\, {\mathfrak 
h}/W$. For a dominant integral coweight $\check\lambda \,\in\, {\mathfrak h}$, denote by 
$${\rm Op}^{{\rm RS},\varpi(-\check\lambda-\check\rho)}_{\mathfrak g}$$ the space of 
$G$--opers with regular singularity on a punctured disc with fixed residue 
$\varpi(-\check\lambda-\check\rho)$.

Frenkel and Gaitsgory have defined in \cite[Section 2.9]{FG}, following Beilinson and 
Drinfeld (unpublished), a subspace \begin{equation}\label{eef} {\rm Op}^{\check\lambda,{\rm 
reg}}_{\mathfrak g} \,\subset\, {\rm Op}^{{\rm 
RS},\varpi(-\check\lambda-\check\rho)}_{\mathfrak g} \end{equation} consisting of all those 
$G$--opers whose monodromy around the origin is trivial. So, in particular, ${\rm 
Op}^{0,{\rm reg}}_{\mathfrak g}$ is naturally identified with the space of regular 
$G$--opers on the disc. It may be mentioned that opers of this kind play an important 
r\^ole in the geometric Langlands correspondence (see \cite{FG2} and \cite[Section 
9.6]{Fr2}).

The following proposition is due to E. Frenkel.

\begin{proposition}\label{propef}
There is a natural isomorphism between the space of all branched 
${\rm SL}(r,{\mathbb C})$--opers on $X$ with branching over $S$ and the space of
all ${\rm SL}(r,{\mathbb C})$--opers on $X$ with 
regular singularity at $S$ satisfying the following two conditions:
\begin{enumerate}
\item the restriction to the punctured formal disc around each 
point $x_k \,\in\, S$ has residue $\varpi(-2\check\rho)$, and

\item this oper belongs to the subspace 
${\rm Op}^{\check\rho,{\rm reg}}_{\mathfrak g}$ in \eqref{eef}.
\end{enumerate}
\end{proposition}

The proof of Proposition \ref{propef} follows directly by comparing Definition
\ref{def1} and Theorem \ref{thm1} with the 
definition of the subspace ${\rm Op}^{\check\rho,{\rm reg}}_{\mathfrak g}$ in \cite{FG}.

\section*{Acknowledgements}

We are very grateful to Edward Frenkel for communicating Section \ref{se7}.
We thank the referee for helpful comments.

%%%%%%%%%%%%%%%%%%%%%%%%%%%%%%%%%%%%%%%%%%%%%%%%%%%%%%%%%%%%%%%%%%%%%%%%%%%%%%%%%%%%%%%%

\end{document}